\documentclass[a4paper,11pt,reqno]{amsart}

\usepackage[hmargin=2.5cm,bmargin=2.5cm,tmargin=2.2cm]{geometry}
\usepackage[utf8]{inputenc}
\usepackage{amsmath, amssymb, amsthm}
\usepackage{mathtools, stackrel, bbm}
\usepackage{comment,cite}	
\usepackage[unicode]{hyperref}
\hypersetup{
    colorlinks=true,
    linkcolor=blue,
    citecolor=red,
    filecolor=magenta,      
    urlcolor=cyan,
    pdftitle={OU on graphs},
    bookmarks=true,
    linktocpage=true,
}
\usepackage[linguistics]{forest}
\usepackage[shortlabels]{enumitem}
\usepackage{marginnote}
\usepackage{color}

\usepackage{tikz}
\usepackage{tikz-cd}

\newcommand{\bbC}{\mathbb{C}}
\newcommand{\bbK}{\mathbb{K}}
\newcommand{\bbN}{\mathbb{N}}
\newcommand{\bbR}{\mathbb{R}}

\newcommand{\calF}{\mathcal{F}}

\newcommand{\calH}{\mathcal{H}}
\newcommand{\calL}{\mathcal{L}}
\newcommand{\calO}{\mathcal{O}}
\newcommand{\calS}{\mathcal{S}}
\newcommand{\calT}{\mathcal{T}}
\newcommand{\calV}{\mathcal{V}}

\DeclareMathOperator{\one}{{\mathbf{1}}} % constant function with value one
\DeclareMathOperator{\zero}{{\mathbf{0}}} % constant function with value zero
\DeclareMathOperator{\re}{Re} % real part
\newcommand{\argument}{\mathord{\,\cdot\,}} % argument dot for functions (with correct spacing)
\newcommand{\dx}{\;\mathrm{d}} % differential (for use at the end of integrals)
\newcommand{\dxShort}{\mathrm{d}} % differential without extra space (for use in differential quotients)
\DeclareMathOperator{\linSpan}{span} % linear span
\DeclareMathOperator{\Fix}{Fix} % fixed space of an operator or semigroup
\DeclareMathOperator{\support}{supp} % support of a function
\newcommand{\norm}[1]{\left\lVert #1 \right\rVert} % norm
\newcommand{\modulus}[1]{\left\lvert #1 \right\rvert} % modulus
\newcommand{\duality}[2]{\left\langle#1\, ,\, #2\right\rangle} % duality / scalar product
\newcommand{\dom}[1]{\operatorname{dom}\left(#1\right)} % domain of an operator
\newcommand\restrict[1]{\raisebox{-.5ex}{$|$}_{#1}} %restriction of a map
\newcommand{\cardinality}[1]{\# #1} %cardinality of a set

\newcommand{\spec}{\sigma} % spectrum
\newcommand{\Graph}{\mathcal{T}}   % tree 
\newcommand{\VertexSet}{\mathsf{V}}    % vertex set
\newcommand{\LeafSet}{\mathsf{L}}    % vertex set
\newcommand{\EdgeSet}{\mathsf{E}}    % edge set
\newcommand{\oVertex}{\mathsf{o}} %vertex o (root)
\newcommand{\vVertex}{\mathsf{v}} %vertex v
\newcommand{\wVertex}{\mathsf{w}} %vertex w
\newcommand{\zVertex}{\mathsf{z}} %vertex z
\newcommand{\eEdge}{\mathsf{e}} %edge e
\newcommand{\edge}[2]{\left[#1\, ,\, #2\right]} % edge [v,w]
\newcommand{\halflines}{m_\infty} % total number of half-lines

\newcommand{\Dirichlet}{\mathrm{D}}
\newcommand{\Neumann}{\mathrm{N}}
\newcommand{\Wentzell}{\mathrm{W}}

\newcommand{\OU}{\mathcal{A}} % OU operator
\newcommand{\OUDir}{\OU^{\Dirichlet}} % OU operator (Dirichlet realisation)
\newcommand{\OUNeu}{\OU^{\Neumann}} % OU operator (Neumann realisation)
\newcommand{\OUWen}{\OU^{\Wentzell}} % OU operator (Wentzell realisation)

\newcommand{\SemiX}{ \left(e^{-t\OU^{\mathrm X}}\right)_{t\ge 0} } % OU semigroup 
\newcommand{\SemiDir}{ \left(e^{-t\OUDir}\right)_{t\ge 0} } % OU semigroup (Dirichlet realisation)
\newcommand{\SemiNeu}{\left(e^{-t\OUNeu}\right)_{t\ge 0} } % OU semigroup (Neumann realisation)
\newcommand{\form}{\mathfrak{a}} % form
\newcommand{\formDir}{\form^{\Dirichlet}} % form (Dirichlet realisation)
\newcommand{\formNeu}{\form^{\Neumann}} % form (Neumann realisation)

\newcommand{\domX}[1]{\operatorname{dom}^{\mathrm X}\left(#1\right)} % domain of X
\newcommand{\domDir}[1]{\operatorname{dom}^{\Dirichlet}\left(#1\right)} % domain of Dirichlet realisation
\newcommand{\domNeu}[1]{\operatorname{dom}^{\Neumann}\left(#1\right)} % domain of Neumann realisation

\newcommand{\OUXLevel}[1]{\mathfrak{A}_{#1}^{\mathrm X}}  % Reduced OU operator
\newcommand{\OUDirLevel}[1]{\mathfrak{A}_{#1}^{\Dirichlet}}  % Reduced OU operator (Dirichlet realisation)
\newcommand{\OUNeuLevel}[1]{\mathfrak{A}_{#1}^{\Neumann}} % Reduced OU operator (Neumann realisation)

\theoremstyle{definition}
\newtheorem{definition}{Definition}[section]
\newtheorem{remark}[definition]{Remark}
\newtheorem{remarks}[definition]{Remarks}
\newtheorem*{remark*}{Remark}
\newtheorem*{remarks*}{Remarks}
\newtheorem{example}[definition]{Example}

\theoremstyle{plain}
\newtheorem{proposition}[definition]{Proposition}
\newtheorem{lemma}[definition]{Lemma}
\newtheorem{theorem}[definition]{Theorem}
\newtheorem{corollary}[definition]{Corollary}

\numberwithin{equation}{section} % enumerate formulas within sections

\begin{document}
    \title[Ornstein-Uhlenbeck on trees]{Ornstein--Uhlenbeck semigroup on rooted trees}
\author[S.~Arora]{Sahiba Arora}
\address{Sahiba Arora, Leibniz Universität Hannover, Institut für Analysis, Welfengarten 1, 30167 Hannover, Germany}
\email{sahiba.arora@math.uni-hannover.de}

\author[M. Kramar Fijavž]{Marjeta Kramar Fijavž}
\address{Marjeta Kramar Fijavž, Faculty of Civil and Geodetic Engineering, University of Ljubljana, Jamova 2, 1000 Ljubljana, Slovenia and Institute of mathematics, physics and mechanics, Jadranska 19, 1000 Ljubljana, Slovenia}
\email{marjeta.kramar@fgg.uni-lj.si}

\author[D.~Mugnolo]{Delio Mugnolo}
\address{Delio Mugnolo, Lehrgebiet Analysis, Fakultät Mathematik und Informatik, FernUniversität in Hagen, 58084 Hagen, Germany}
\email{delio.mugnolo@fernuni-hagen.de}

\author[A.~Rhandi]{Abdelaziz Rhandi}
\address{Abdelaziz Rhandi, Dipartimento di Matematica, Università degli Studi di Salerno, Via Giovanni Paolo II 132, I‑ 84084 Fisciano (SA), Italy}
\email{arhandi@unisa.it}

\subjclass[2020]{Primary 35P20; Secondary 34B45, 47D07, 35R02, 31C25}

\keywords{Ornstein–Uhlenbeck operators, infinite metric graphs, $C_0$-semigroups, Dirichlet forms, invariant measure, discrete spectrum}
\date{\today}
\begin{abstract}
    We study Ornstein--Uhlenbeck operators on rooted metric trees equipped with a Gaussian-type measure. Using form methods, we construct Dirichlet and Neumann realisations corresponding, respectively, to killing and reflection at the root. The associated semigroups are symmetric, analytic and positivity preserving; the Dirichlet semigroup is sub-Markovian, while the Neumann semigroup is Markovian and admits the Gaussian measure as its unique invariant measure up to scalar multiples. We prove compactness of the resolvent and derive linear eigenvalue asymptotics. For regular rooted trees, we adapt the Naimark--Solomyak decomposition to the Gaussian weighted setting, reducing the operators to one-dimensional half-line problems and obtaining refined spectral localisation and lower bounds.
\end{abstract}

\maketitle
\section{Introduction}

The purpose of this article is to study the Ornstein--Uhlenbeck operator
\[
    (\OU f)(x) = -\frac{1}{2} f''(x) + \modulus{x}f'(x)
\]
on rooted metric trees with compact core and finitely many infinite edges. The choice of a root is essential:~it provides the radial coordinate entering the unbounded drift term.

The Ornstein--Uhlenbeck operator is a classical prototype of a second-order
operator with unbounded drift. Originating in the mathematical description of
the Ornstein--Uhlenbeck stochastic process, it has become a standard model in
the theory of diffusion semigroups, invariant Gaussian measures, and elliptic
operators with unbounded coefficients; see, for instance,
\cite{Lorenzi2017,MetafunePrussRhandiSchnaubelt2002, MetafunePallaraPriola2002, Lunardi1997}. 
Its study on non-compact metric graphs was initiated in
\cite{MugnoloRhandi2022}, where Ornstein--Uhlenbeck semigroups on metric star
graphs were constructed and analysed explicitly.
There, the authors considered a metric star graph consisting of finitely many half-lines and
showed that a natural Gaussian measure gives rise to a self-adjoint Ornstein--Uhlenbeck operator with standard transmission conditions at the central vertex. They proved generation of a symmetric analytic Markov semigroup, identified the invariant measure, obtained an explicit heat kernel, and described the spectrum in terms of the number of edges. Our aim is to
extend part of this picture from star graphs to rooted metric trees.

The general framework is that of metric graphs as metric measure spaces. In the classical setting, one equips the graph with the Euclidean distance along edges and with the Lebesgue measure. The corresponding Dirichlet form then leads to the usual Kirchhoff Laplacian. As already observed in \cite{MugnoloRhandi2022}, changing the underlying measure changes the
associated differential operator. In the present paper we equip $\Graph$ with the Gaussian-type measure
\[
    \mu(\dxShort x)=G(\modulus{x})\dx x;
\]
where $G$ is the one-dimensional Gaussian density. The associated weighted Dirichlet form naturally gives rise to the Ornstein--Uhlenbeck expression on the edges, together with suitable transmission conditions at the vertices.

We study the Dirichlet and Neumann realisations, distinguished by the boundary condition at the root, corresponding respectively to killed and conservative dynamics. This framework yields the basic semigroup theory, compact resolvent, and the invariant-measure result for
the Neumann realisation.

A central theme of the paper is spectral asymptotics. In the general case, we prove that the spectrum is discrete and
that the eigenvalues grow linearly. The leading asymptotic coefficient depends only on the number of infinite edges, reflecting the fact that the high-energy
behaviour is determined by the geometry at infinity.

For regular rooted trees, we obtain more refined information by adapting
decomposition methods for radial operators on regular trees \cite{Carlson2000, NaimarkSolomyak2000, NaimarkSolomyak2001,Solomyak2004, SobolevSolomyak2002}. This reduces the analysis to a family of one-dimensional Ornstein--Uhlenbeck operators on half-lines, with the branching structure encoded by the radial branching function. Although a full
spectral description as in the star-graph case is no longer available, this decomposition yields effective localisation estimates and explicit lower bounds for the eigenvalues.

\subsection*{Outline of the article}

The article is organised as follows. In Section~\ref{sec:rooted-trees} we recall the relevant metric graph formalism, introduce rooted metric trees, and define the weighted Sobolev spaces used throughout the paper. In Section~\ref{sec:ou-general} we construct the Dirichlet and Neumann realisations of the Ornstein--Uhlenbeck operator by form methods, identify their operator domains, and establish their basic semigroup properties and spectral asymptotics; we also briefly discuss a Wentzell-type realisation at the root. In Section~\ref{sec:regular-trees} we restrict to regular rooted trees, adapt the Naimark--Solomyak decomposition, and prove refined spectral bounds. Finally, the Appendix~\ref{appendix:half-line} collects the one-dimensional half-line eigenvalue asymptotics needed in the main text.

\section{Rooted Metric Trees and Function Spaces}
    \label{sec:rooted-trees}

In this section we introduce the geometric and functional analytic framework used throughout the paper. We fix a rooted metric tree, define the associated Gaussian measure, and introduce the relevant function spaces.

\subsection{Rooted metric trees}

Let $\widetilde{\Graph}$ be a compact, rooted metric tree with vertex set $\VertexSet$. We denote by $\oVertex \in \VertexSet$, the root of the tree and by $\emptyset\ne \LeafSet\subset\VertexSet \setminus\{\oVertex\}$ the subset of all leaves, i.e., all vertices of degree one excluding the root. Our object of interest is the rooted tree graph ${\Graph}$, obtained from $\widetilde{\Graph}$ by attaching, to each vertex in
$\LeafSet$, a finite non-zero number of half-lines.
This construction allows us to model infinite trees obtained by attaching finitely many half-lines to boundary vertices.
The edge set of $\Graph$ is denoted by $\EdgeSet$ and we write $\EdgeSet_\infty\subset\EdgeSet$ to denote the set of half-lines attached to the leaves of $\widetilde{\Graph}$ with
\[
    \halflines\coloneqq \cardinality{\EdgeSet_\infty}.
\]

As usual, the edges of the tree are treated as non-degenerate line segments, i.e., each $\eEdge\in \EdgeSet$ is assigned a length $\ell_\eEdge>0$.
Note that by $x\in\Graph$, we mean any point on a line segment representing an edge, not necessarily a vertex.
Because $\Graph$ is a tree, two points $x,y\in \Graph$ have a unique ancestor at maximal distance from the root, which we denote by $x\wedge y$. If $x\wedge y=x$, we say that $y$ is \emph{below} $x$ in $\Graph$, and we write $x \preceq y$; if $x\preceq y$ but $x\ne y$, then we write $x \prec y$. Observe that the path connecting $x$ and $x\wedge y$ is homeomorphic to a real interval.
The shortest path metric on $\Graph$ induced by the Euclidean distance on $\bbR$ is then naturally defined as
\[
    d(x,y)\coloneqq \modulus x+ \modulus y - 2\modulus{x\wedge y},
\]
where we denote 
\[
    \modulus{x}\coloneqq d(x,\oVertex)
\]
for each $x\in \Graph$. Observe that $\modulus y= \modulus x+d(x,y)$ if $x\preceq y$. In particular, our assumptions enforce that the \emph{radius} of the tree 
$
    r(\Graph)\coloneqq \sup_{x \in \Graph}\modulus x = \infty.
$

If $\vVertex$ and $\wVertex$ are respectively the \emph{initial} and \emph{terminating} points of an edge $\eEdge$, then we write
\[
    \eEdge = \edge{\vVertex}{\wVertex} \coloneqq \{ x \in \Graph : \vVertex \preceq  x \preceq \wVertex\}
\]
and say that $\eEdge$ is \emph{emanating} from the vertex $\vVertex$. By $\modulus{\eEdge}$ we mean the length of the edge $\eEdge$, i.e., $\modulus{\eEdge}=d(\vVertex,\wVertex)$. Further, each edge $\eEdge$ incident to a vertex $\vVertex$ is denoted by $\eEdge \sim \vVertex$.
We fix the following convention for edge parametrisations:~An edge $\eEdge = \edge{\vVertex}{\wVertex}$ is parametrised by the arc-length coordinate $s\in[0,\modulus{\eEdge}]$ with $s=0$ corresponding to $\vVertex$ and $s=\modulus{\eEdge}$ to $\wVertex$. 

\subsection{Function spaces on the tree}

Let $\Graph= (\VertexSet, \EdgeSet)$ be a rooted tree.
Following \cite{Mug19}, a metric measure structure on $\Graph$ naturally induces the relevant function spaces on the graph. In particular, we consider the spaces of continuous functions and compactly supported continuous functions
\[
    \mathrm C(\Graph),\quad \mathrm C_{\mathrm c}(\Graph)
\]
as well as the Hilbert space
\[
    L^2_\mu(\Graph)\coloneqq\bigoplus_{\eEdge \in \EdgeSet}L^2_\mu(\eEdge),
\]
defined with respect to a given measure $\mu$ on $\Graph$. Likewise, for $k\in\bbN$, we introduce the edge-wise Sobolev spaces
\[
    \widetilde{H^k_\mu} (\Graph)\coloneqq\bigoplus_{\eEdge \in \EdgeSet}H^k_\mu(\eEdge),\quad k\in \bbN,
\]
and define the first-order Sobolev space on the graph by imposing continuity across vertices:
\[
    H^1_{\mu}(\Graph)\coloneqq\widetilde{H^1_\mu}(\Graph)\cap \mathrm C(\Graph).
\]
Note that derivatives are always understood in the weak sense on each edge.
Consider an edge $\eEdge= \edge{\vVertex}{\wVertex}$ and let $s\in [0, \modulus \eEdge]$. For each function $f$ on $\eEdge$, we denote by $f_{\eEdge}'(s)$ its derivative with respect to this coordinate. The \emph{outward normal derivative} of $f$ at a vertex $\vVertex$ along the edge $\eEdge$ incident to $\vVertex$ (pointing away from $\vVertex$) is then
\[
    \frac{\partial f_{\eEdge}}{\partial n}(\vVertex) := 
        \begin{cases}
            f_{\eEdge}'(0), \qquad & \text{if }\eEdge \text{ is emanating from }\vVertex\\ 
           - f_{\eEdge}'(\modulus{\eEdge}), \qquad & \text{if }\eEdge \text{ terminates at }\vVertex. 
        \end{cases}
\]

Finally, integration over the graph is understood in the natural edge-wise sense, i.e.,
\[
    \int_{\Graph} f~\mu(\dxShort x) \coloneqq \sum_{\eEdge \in \EdgeSet} \int_{\eEdge} f~\mu(\dxShort x).
\]

Given an arbitrary Borel measure on $\bbR$, we can lift it to a measure on  $\Graph$ defined as the direct sum measure of its restrictions to each edge. In the context of Ornstein--Uhlenbeck operators, the most canonical choices are the Lebesgue measure and the Gaussian measure $\mu$:~the latter is the measure whose density with respect to the Lebesgue measure is the Gaussian function, i.e., 
\begin{equation}
    \label{eq:gaussian}
    \mu(\dxShort x)\coloneqq G(x)\dx x\coloneqq  \frac{2}{\sqrt{\pi}} e^{-\modulus{x}^2} \dx x, \quad x\in\Graph.    
\end{equation}
Being absolutely continuous with respect to the Lebesgue measure, $\mu(\dxShort x)$ is locally finite with respect to $d$. Note that the density of the Gaussian measure depends only on the distance from the root. This reflects the radial nature of the Ornstein--Uhlenbeck drift.

Recall that each edge $\eEdge \in \EdgeSet$ is represented by a line segment $[s,t]$ or $[s,\infty)$. In the sequel, for the sake of convenience, we accordingly use the notation
\[
     f(x)\bigg|_{\eEdge}\coloneqq  f(x)\bigg|_{s}^{t} \quad \text{or}\quad f(x)\bigg|_{\eEdge}\coloneqq f(x)\bigg|_{s}^{\infty}.
\]

\begin{lemma}\label{lem:integration-by-parts}
    On each edge $\eEdge\in \EdgeSet$, 
    the Gaussian measure~\eqref{eq:gaussian} satisfies the integration by parts formula
    \begin{equation}
        \label{eq:integration-by-parts-gaussian}
        \int_{\eEdge} f(x)h'(x) \mu(\dxShort x) = f(x)h(x)G(x)\bigg|_{\eEdge} - \int_{\eEdge} \big(f'(x)-2\modulus{x} f(x)\big) h(x)\mu(\dxShort x)
    \end{equation}
    for all $f,h\in H_{\mu}^1(\Graph)$; where $G$ denotes the Gaussian function defined in \eqref{eq:gaussian}.
\end{lemma}

\begin{proof}
    Let $f,h\in H_{\mu}^1(\Graph)$. Since 
    \[
        \int_{\eEdge} f(x)h'(x) \mu(\dxShort x) = \int_{\eEdge} f(x)h'(x) G(x)\dx x,
    \]
    the integration by parts formula for the Lebesgue measure and $G'(x)=-2\modulus{x}G(x)$ yield that
    \begin{align*}
        \int_{\eEdge} f(x)h'(x) \mu(\dxShort x) & = f(x)h(x)G(x)\bigg|_{\eEdge} - \int_{\eEdge} \big(f'(x)G(x)-2\modulus{x}f(x)G(x)\big) h(x) \dx x\\
                                                & =  f(x)h(x)G(x)\bigg|_{\eEdge} - \int_{\eEdge} \big(f'(x)-2\modulus{x} f(x)\big) h(x)\mu(\dxShort x),
    \end{align*}
    as desired.
\end{proof}

\section{Ornstein--Uhlenbeck Realisations and Semigroup Properties}
    \label{sec:ou-general}

In this section we introduce the Ornstein--Uhlenbeck operators on the rooted tree $\Graph$ via quadratic form methods. We define Dirichlet and Neumann realisations, characterise their domains, and establish their basic semigroup properties.

The \emph{Ornstein–Uhlenbeck operator on the rooted tree} $\Graph= (\VertexSet, \EdgeSet)$ is defined as
\[
    (\OU f)(x) = -\frac{1}{2} f''(x) + \modulus{x}f'(x)\qquad (x\notin \VertexSet).
\]

We now define two realisations of the Ornstein--Uhlenbeck expression by imposing vertex conditions. Let $f$ be a function on $\Graph$ and $\vVertex \in \VertexSet$. We write $f_{k}$ to denote the restriction of $f$ to an edge $\eEdge_k$ incident to $\vVertex$. At $\vVertex\in\VertexSet$, we consider the boundary conditions
\begin{equation}
    \label{eq:BC}\tag{BC$_{\vVertex}$}
    \begin{aligned}
        f_{j}(\vVertex) = f_k(\vVertex)&
        \quad  &\text{for }\eEdge_j, \eEdge_k \sim \vVertex,\\
        \sum_{\eEdge_k \sim \vVertex  } \frac{\partial f_k}{\partial n}(\vVertex)  =0.&
    \end{aligned}
\end{equation}
These are the usual \emph{continuity} and \emph{Kirchhoff} boundary conditions, respectively; jointly also called the \emph{standard conditions}.
Note that a function lies in $H^1_\mu(\Graph)$ if it lies in $\widetilde{H^1_\mu}(\Graph)$ and it satisfies the first condition in~\eqref{eq:BC} for all $\vVertex\in\VertexSet$.

\subsection{The Dirichlet and Neumann realisations}
    \label{sec:dirichlet-neumann-realisation}

On $\Graph$, the \emph{Dirichlet realisation} of the Ornstein--Uhlenbeck operator is given by the operator $\OUDir \coloneqq (\OU, \domDir{\OU})$, where
\[
    \domDir{\OU} \coloneqq \left\{ f\in \widetilde{H_{\mu}^2}(\Graph) : f(\oVertex)=0 \text{ and } f \text{ satisfies}~\eqref{eq:BC} \text{ for all } \vVertex\in\VertexSet \setminus \{\oVertex\}\right\}.
\]
Analogously, \emph{Neumann realisation} of Ornstein--Uhlenbeck operator is given by the operator $\OUNeu \coloneqq (\OU, \domNeu{\OU})$, where
\[
    \domNeu{\OU} \coloneqq \left\{ f\in \widetilde{H_{\mu}^2}(\Graph) :  f \text{ satisfies}~\eqref{eq:BC}\text{ for all }\vVertex\in\VertexSet\right\}.
\]

\subsubsection{Quadratic forms and operator identification}

We apply form methods to show that both realisations of $\OU$ generate a $C_0$-semigroup on $L_{\mu}^2(\Graph)$.
First, we introduce a sesquilinear form $\form(\argument,\argument)$ given by
\[
    \form(f,g)\coloneqq \frac{1}{2} \int_{\Graph} f'(x)\overline{g'(x)} \mu(\dxShort x).
\]
The Dirichlet and Neumann forms are given by restricting $\form$ to the domains
\begin{align*}
    \domDir{\form}\coloneqq H^1_{0,\mu}(\Graph) &\coloneqq \{f\in H^1_\mu(\Graph): f(\oVertex)=0\}
    \text{ and}\\
    \domNeu{\form}\coloneqq H^1_{\mu}(\Graph)
\end{align*}
respectively. We use the shorthand $\formDir \coloneqq (\form, \domDir{\form})$ and $\formNeu \coloneqq (\form, \domNeu{\form})$.

Before we relate the above operators and forms in the obvious way, let us record a few basic properties of the two forms.
We refer to \cite{FukOshTak10,Ouh05} for all the terminology regarding Dirichlet forms and their associated semigroups. In particular, a symmetric form $(\mathfrak q, \dom{\mathfrak q})$ on $L^2_{\mu}(\Graph)$ is said to be \emph{regular} if the space $\dom{\mathfrak q} \cap \mathrm C_{\mathrm c}(\Graph)$ is dense in $\dom{\mathfrak q}$ with the form norm and dense in $\mathrm C_{\mathrm c}(\Graph)$ with the uniform norm. Also, $(\mathfrak q, \dom{\mathfrak q})$ is called \emph{strongly local} if for compactly supported functions $f,g\in \dom{\mathfrak q}$ whenever $f$ is constant in a neighbourhood of $\support g$, then $\mathfrak q(f,g)=0$.

\begin{proposition}
    \label{prop:form-properties}
    The forms $\formDir$ and $\formNeu$ are closed, densely defined, symmetric, non-negative, and strongly local Dirichlet forms on $L^2_\mu(\Graph)$. In addition, $\formDir$ is regular.

    Furthermore, the form domain $\domNeu{\form}$ and, in turn, $\domDir{\form}$, is compactly embedded in $L^2_\mu(\Graph)$.
\end{proposition}

\begin{proof}
    The symmetry and non-negativity of the forms are immediate from the definition, while the argument for closedness is standard.

    \emph{Dense domain}:
    Since $\Graph$ is a metric graph, the space of functions that are smooth on each edge with compact support on the edges is dense in $L^2_\mu(\Graph)$. Such functions belong to $H^1_\mu(\Graph)$, hence $\domNeu{\form}$ is dense. Moreover, as the root $\oVertex$ has $\mu$-measure zero, these functions may be chosen to vanish at $\oVertex$, and thus lie in $H^1_{0,\mu}(\Graph)$; consequently, $\domDir{\form}$ is also dense in $L^2_\mu(\Graph)$.

    \emph{Compact domain}:
    Since $\Graph$ consists of a compact metric graph with finitely many half-lines attached, compactness is local apart from the behaviour at infinity.
    On the compact part, the embedding follows from the usual Rellich theorem.
    On each half-line, the Gaussian tail estimate used in the proof of
    \cite[Theorem~10.3.17]{Lorenzi2017} gives uniform $L^2_\mu$-tightness
    for bounded subsets of $H^1_\mu$. Together with Rellich compactness on finite
    subintervals, this yields compactness on each half-line. Since the number of
    half-lines is finite, $H^1_\mu(\Graph)$ is compactly embedded in $L^2_\mu(\Graph)$.

    \emph{Strongly local Dirichlet forms}:
    Let $f\in H^1_{\mu}(\Graph)$ and let $g$ be a normal contraction of $f$, i.e., 
    \[
        \modulus{g(x)-g(y)} \le \modulus{f(x)-f(y)}
        \text{ and }
        \modulus{g(x)} \le \modulus{f(x)}
        \qquad\text{for all }x,y\in \Graph.
    \]
    Then $g$ is also continuous. Moreover, by the one-dimensional contraction principle on each edge $\eEdge$, we have that $g_{\eEdge}\in H^1_\mu(\eEdge)$ and $\modulus{g_{\eEdge}'}\le \modulus{f_{\eEdge}'}$. In particular, $g\in H^1_\mu(\Graph)$ and $\form(g,g)\le \form(f,f)$.
    In addition, if $f(\oVertex)=0$, then $g(\oVertex)=0$.
    Thus both forms are Dirichlet forms.
    Finally, if $f$ is constant on a neighbourhood of $\support g$, then $f_{\eEdge}'=0$ almost everywhere on $\support g\cap \eEdge$ for all $\eEdge$, and in turn, $\form(f,g)=0$. Hence, the forms are strongly local.

    \emph{Regularity of $\formDir$}:
    By standard edgewise approximation, the space $\mathrm{Lip}_{\mathrm c}(\Graph)$ of all Lipschitz continuous functions that have compact support in the interior of $\Graph$ is dense in $H^1_{0,\mu}(\Graph)$ with respect to the form norm. Moreover, $\mathrm{Lip}_{\mathrm c}(\Graph)$ is uniformly dense in $\mathrm C_{\mathrm c}(\Graph)$, by edgewise piecewise affine approximation.
    Hence they form a core of $\formDir$, and therefore $\formDir$ is regular.
\end{proof}

We now relate the abstract operators associated with $\formDir$ and $\formNeu$ to the concrete Dirichlet and Neumann realisations of the Ornstein--Uhlenbeck operator defined above.

\begin{proposition}
    \label{prop:form-operator}
    The self-adjoint operator associated with the form $\formDir$ is precisely the Dirichlet realisation $\OUDir$ of the Ornstein--Uhlenbeck operator. Likewise, the self-adjoint operator associated with $\formNeu$ is precisely the Neumann realisation $\OUNeu$.
\end{proposition}

\begin{proof}
    We prove the assertion for the Dirichlet realisation, the Neumann case being analogous. Let $B^{\Dirichlet}$ denote the operator associated with $\formDir$. By definition,
    \begin{align*}
        \dom{B^{\Dirichlet}} &=\left\{f\in \domDir{\form} :\exists~g\in L^2_{\mu}(\Graph) \text{ with } \form(f,h)=\duality{g}{h}_{L_{\mu}^2(\Graph)}\forall~h\in \domDir{\form}\right\},\\
        B^{\Dirichlet}f      &=g.
    \end{align*}

    First, fix $f\in \domDir{\OU}$. Then $f\in \domDir{\form}$ and integration by parts formula~\eqref{eq:integration-by-parts-gaussian} gives for every $h\in \domDir{\form}$,
    \begin{align*}
        2\form(f,h) = \sum_{\eEdge\in\EdgeSet} \int_{\eEdge} f'(x)\overline{h'(x)} \mu(\dxShort x)
                   = \sum_{\eEdge\in\EdgeSet}\left[ f'(x)\overline{h(x)} G(x)\bigg|_{\eEdge} + 2\int_{\eEdge} (\OU f)(x)\overline{h(x)}\mu(\dxShort x)\right]
    \end{align*}
    Rearranging the boundary sum on the right-hand side vertex-wise and appealing to the vertex conditions fulfilled by
    $h\in \domDir{\form}$ and $f\in \domDir{\OU}$,
    it follows that $\form(f,h) = \duality{\OU f}{h}_{L_{\mu}^2(\Graph)}$. In turn, $\OUDir\subseteq B^{\Dirichlet}$.

    Conversely, let $f\in \dom{B^{\Dirichlet}}$. Then $f(\oVertex)=0, f$ is continuous at each vertex, and
    there exists $g\in L_{\mu}^2(\Graph)$
    such that
    \begin{equation}
	   \label{eq:form-formula}
	   \form(f,h)=\duality{g}{h}_{L_{\mu}^2(\Graph)} 
       \quad \text{for all } h\in \domDir{\form}.
    \end{equation}
    Fix $\eEdge\in\EdgeSet$ and choose $h\in \domDir{\form}$ such that $h$ is a smooth function with compact support on $\eEdge$ and $0$ elsewhere.
    By~\eqref{eq:form-formula}, integration by parts \eqref{eq:integration-by-parts-gaussian} on $\eEdge$, and the standard one-dimensional characterisation of the operator associated with the weighted Sobolev form on an interval, we obtain $f\restrict{\eEdge}\in H_{\mu}^2 (\eEdge)$ and $-\frac{1}{2} f''(x) + \modulus{x}f'(x)=g(x)$ whenever 
    $x\in \eEdge$. It follows that $f\in \widetilde{H_{\mu}^2(\Graph)}$ which allows us to integrate by parts in~\eqref{eq:form-formula}, as above, to get
    \begin{equation}
	\label{eq:form-formula-bis}
        \form(f,h) = \frac{1}{2} \sum_{\eEdge\in\EdgeSet} f'(x)\overline{h(x)} G(x)\bigg|_{\eEdge} + \duality{g}{h}_{L_{\mu}^2(\Graph)}
        \quad \text{for all } h\in \domDir{\form}.
    \end{equation}
    Comparing with~\eqref{eq:form-formula} and employing the surjectivity of the trace map
    \[
        \domDir{\form}\ni h\mapsto (h(\vVertex))_{\vVertex \in \VertexSet}\in \{ (a_{\vVertex})_{\vVertex \in\VertexSet}: a_{\oVertex}=0\},
    \]
    we deduce that the boundary term in~\eqref{eq:form-formula-bis} must vanish identically and, hence, that $f$ also satisfies the Kirchhoff condition from~\eqref{eq:BC} for all $\vVertex \in \VertexSet\setminus \{\oVertex\}$. It follows that $f\in \domDir{\OU}$. Coupling with $\OU f=g=B^{\Dirichlet}f$, we conclude that $\OUDir\supseteq B^{\Dirichlet}$.
\end{proof}

As a direct consequence of the compact embedding of the form domains, we obtain the following spectral property.

\begin{corollary}
    \label{cor:compact-resolvent}
    The operators $\OUDir$ and $\OUNeu$ have compact resolvent and hence pure point spectrum.
    Moreover, $\ker \OUNeu = \linSpan\{\one\}$ and $\OUDir$ is invertible.
\end{corollary}

\begin{proof}
    By the standard correspondence between closed forms and self-adjoint operators and Proposition~\ref{prop:form-operator}, compactness of the form-domain embedding (Proposition~\ref{prop:form-properties}) implies compact resolvent and pure point spectrum of both the operators.

    Now, observe that the constant functions lie in $L_{\mu}^2(\Graph)$ and hence in $\ker \OUNeu$; and indeed, it spans it as any function that annihilates the quadratic form must be constant. 
    On the other hand, as $\domDir{\OU}$ contains no non-trivial constant function, $\OUDir$ must be invertible.
\end{proof}

\subsubsection{Semigroup properties and invariant measure}

Proposition~\ref{prop:form-operator} and the standard theory of forms ensure that the self-adjoint operators $-\OUDir$ and $-\OUNeu$ generate $C_0$-semigroups on $L^2_\mu(\Graph)$. In the following,  we derive the corresponding semigroup properties.

\begin{proposition}
    \label{prop:semigroups}
    Let $\SemiDir$ and $\SemiNeu$ denote the $C_0$-semigroups on $L^2_\mu(\Graph)$ generated by the operators associated with $\formDir$ and $\formNeu$, respectively. Then the following assertions hold.
    
    \begin{enumerate}[\upshape (a)]
        \item \label{prop:semigroups:itm:analytic}
        Both semigroups are contractive, analytic, and compact on $L^2_\mu(\Graph)$.
    
        \item \label{prop:semigroups:itm:markov}
        The semigroup $\SemiDir$ is sub-Markovian, whereas $\SemiNeu$ is Markovian.

        \item \label{prop:semigroups:itm:strong-Feller}
        Both semigroups are strong Feller.
    
        \item \label{prop:semigroups:itm:lp}
        Both semigroups admit consistent extensions to contractive $C_0$-semigroups on $L^p_\mu(\Graph)$ for all $1\le p<\infty$.
    
        \item \label{prop:semigroups:itm:irreducible}
        The Neumann semigroup is always irreducible while the Dirichlet semigroup is irreducible if $\deg{\oVertex}=1$.
    
        \item \label{prop:semigroups:itm:domination}
        The Dirichlet semigroup is dominated by the Neumann semigroup, i.e.,
        \[
            0\le e^{-t\OUDir}f \le e^{-t\OUNeu}f
            \quad \text{for all } t\ge 0 \text{ and } 0 \le f\in L^2_\mu(\Graph).
        \]
    
        \item \label{prop:semigroups:itm:transient-recurrent}
        The Dirichlet form $\formDir$ is transient, whereas the Neumann form $\formNeu$ is recurrent.
    \end{enumerate}
\end{proposition}

The distinction between transience and recurrence is consistent with the geometry induced by the Ornstein--Uhlenbeck drift: in the Neumann case, the drift points towards the root without loss of mass at the boundary, whereas in the Dirichlet case mass is absorbed at the root.

\begin{proof}[Proof of Proposition~\ref{prop:semigroups}]
    (a) By the standard theory of forms, both semigroups are contractive and analytic; see \cite[Proposition~1.51 and Theorem~1.52]{Ouh05}. Moreover, their compactness is a consequence of the compact embedding in Proposition~\ref{prop:form-properties}.

    (b) Since both $\formDir$ and $\formNeu$ are symmetric Dirichlet forms (Proposition~\ref{prop:form-properties}), the associated semigroups are sub-Markovian \cite[Theorem~2.25]{Ouh05}.
    Moreover, constant functions lie in $L_{\mu}^2(\calT)$, and hence in the null space of $-\OUNeu$.
    In particular, $e^{-t\OUNeu}\one = \one$ for all $t\ge 0$ and hence $\SemiNeu$ is Markovian.

    (c) Let $\mathrm X\in \{\Dirichlet,\Neumann\}$. As $\SemiX$ is analytic,
    \begin{equation}
        \label{eq:invariant-measure:feller-kernel}
        e^{-t \OU^{\mathrm X}} L^2_\mu(\Graph)\subseteq \domX{-\OU} \subseteq H^1_\mu(\Graph)\subseteq \mathrm C(\Graph) \cap L^\infty_{\mathrm{loc},\mu }(\Graph)
        \quad \text{for all }t>0.
    \end{equation}
    Therefore, $\SemiX$ consists of integral operators by \cite[Proposition~1.7]{AB94}. Hence, for each $x\in \Graph$ and $t>0$, the functional
    $
        f\mapsto e^{-t \OU^{\mathrm X}}f(x)
    $
    is given by integration against a finite Borel measure.

    Moreover, since $\mu(\Graph)<\infty$, we have $L^\infty_\mu(\Graph)\subset L^2_\mu(\Graph)$. Thus~\eqref{eq:invariant-measure:feller-kernel} gives continuity of $e^{-t\OU^{\mathrm X}}f$ for $f\in L^\infty_\mu(\Graph)$. The (sub-)Markovian property gives
    \[
        \norm{e^{-t\OU^{\mathrm X}}f}_{L^\infty_\mu}
        \le
        \norm{f}_{L^\infty_\mu}
    \]
    Since $\mu$ has full support, the continuous representative is bounded
    pointwise by the same constant. 
    Hence $e^{-t \OU^{\mathrm X}} L^\infty_\mu(\Graph)\subseteq \mathrm C_{\mathrm b}(\Graph)$, the space of bounded continuous functions on $\Graph$, i.e., $\SemiX$ is strong Feller.

    (d) Since the (contraction) semigroups are sub-Markovian, they are in particular $L^\infty$-contractive. Symmetry of the forms
    therefore ensures that both semigroups extend to contractive $C_0$-semigroups on $L^p_\mu(\Graph)$ for all $1\le p<\infty$ by standard interpolation arguments, see, for example \cite[Pages~56-57]{Ouh05}.

    (e) We use condition~(2) of \cite[Corollary~2.11]{Ouh05}. Let $\mathrm X \in \{\Dirichlet, \Neumann\}$ and let $\Omega\subseteq\Graph$ be measurable such that
    $
        \one_\Omega \domX{\form}\subset \domX{\form}.
    $
    
    If $\mathrm X=\Neumann$, then
    $\one_{\Omega}=\one_{\Omega}\cdot \one_\Graph \in H^1_{\mu}(\Graph)$. By continuity of $H^1_{\mu}(\Graph)$ functions, $\one_{\Omega}$ must be constant on the connected graph $\Graph$. It follows that either $\mu(\Omega)=0$ or $\mu(\Graph\setminus\Omega)=0$. 
    Thus the Neumann semigroup is irreducible.
    
    On the other hand, if $\mathrm X= \Dirichlet$, then we choose
    $h\in\domDir{\form}$ such that $h>0$ on $\Graph\setminus\{\oVertex\}$. 
    By assumption, $\one_\Omega h\in\domDir{\form}\subset H^1_\mu(\Graph)$, and is hence continuous. Since $h$ is continuous
    and strictly positive on $\Graph\setminus\{\oVertex\}$, it follows that $\one_\Omega$ is constant on the set $\Graph\setminus\{\oVertex\}$. 
    If $\deg{\oVertex}=1$, then $\Graph\setminus\{\oVertex\}$ is connected and it follows that $\mu(\Omega)=0$ or $\mu(\Graph\setminus\Omega)=0$. Consequently, the Dirichlet semigroup is irreducible.

    (f) As $\domDir{\form}$ is an ideal of $\domNeu{\form}$, the asserted dominance is a consequence of \cite[Corollary~2.22]{Ouh05}.

    (g) Since $\OUDir$ is invertible (Corollary~\ref{cor:compact-resolvent}), its first eigenvalue satisfies $\lambda_1^{\Dirichlet}>0$, and hence
    $
        \formDir(f,f)\ge 
        \lambda_1^{\Dirichlet}\norm{f}_{L^2_\mu(\Graph)}^2
    $
    for all $f\in \domDir{\form}$.
    As $\mu(\Graph)<\infty$, Cauchy--Schwarz implies
    \[
        \int_\Graph \modulus{f}\mu(\dxShort x)
        \le \mu(\Graph)^{1/2}\norm{f}_{L^2_\mu(\Graph)}
        \le \mu(\Graph)^{1/2}(\lambda_1^{\Dirichlet})^{-1/2}\formDir(f,f)^{1/2}.
    \]
    Thus $g\coloneqq\sqrt{\lambda_1^{\Dirichlet}\mu(\Graph)^{-1}}\one$ is a bounded integrable reference function, so $\formDir$ is transient by \cite[Theorem~1.5.1]{FukOshTak10}.
    On the other hand, the extended Dirichlet space of $\formNeu$
    contains the function $\one$ that satisfies $\formNeu(\one,\one)=0$. Hence $\formNeu$ is recurrent by \cite[Theorem~1.6.3]{FukOshTak10}.
\end{proof}

We conclude the section by identifying the invariant measure for the Neumann semigroup and describing its long-time behaviour.

\begin{theorem}
    \label{thm:invariant-measure}
    The measure $\mu$ given by \eqref{eq:gaussian} is the unique (up to a multiplicative constant) invariant measure for $\SemiNeu$. Moreover,
    \begin{equation}
        \label{eq:Doobs}
        \lim_{t\to \infty}e^{-t\OUNeu} f(x)=\frac{1}{\mu(\Graph)}\int_{\Graph} f\mu(\dxShort x)
        \quad \text{for all }f\in L^\infty_{\mu}(\Graph) \text{ and }x\in \Graph.
    \end{equation}
\end{theorem}

\begin{proof}
    As $\mu(\Graph)<\infty, \one \in \domNeu{\form}$, and $\form(f,\one)=0$ for all $f\in \domNeu{\form}$, invariance of measure is guaranteed by \cite[Remark~5.1]{VogVoi03}.
    
    Recall from Proposition~\ref{prop:semigroups}\ref{prop:semigroups:itm:strong-Feller} that $\SemiNeu$ is strong Feller. More precisely, using the kernel
    representatives of the operators $e^{-t\OUNeu}$, $t>0$, the semigroup
    extends to a positive bounded strong Feller semigroup on
    $\mathrm B_b(\Graph)$, the space of bounded Borel measurable functions on $\Graph$.
    The measure $\mu$ is finite, invariant, and strictly
    positive. Furthermore, by Corollary~\ref{cor:compact-resolvent}, the fixed space of the semigroup $\SemiNeu$ is given by $\Fix \left(\SemiNeu\right) = \linSpan\{\one\}$. Thus, all the assumptions of \cite[Theorem~4.6]{Gerlach-Nittka} are fulfilled.
    Consequently, there exists $e\in \Fix \left(\SemiNeu\right) = \linSpan\{\one\}$ and
    \[
        \lim_{t\to\infty} e^{-t \OUNeu}f
        =
        \left(\int_\Graph f\mu(\dxShort x)\right)e
        \quad \text{for all } f\in L^\infty_{\mu}(\Graph) \text{ and } x\in \Graph.
    \]
    In particular, taking $e=c\one$ and $f=\one$, we obtain that $c=\mu(\Graph)^{-1}$, and \eqref{eq:Doobs} follows.

    Finally, let $\sigma$ be a finite invariant measure for $\SemiNeu$. Uniqueness of invariant probability measures \cite[Proposition~9.1.15]{Lorenzi2017} ensures  $\mu(\Graph)^{-1} \mu=\sigma(\Graph)^{-1} \sigma$ and in turn, 
    $\sigma=\mu(\Graph)^{-1} \sigma(\Graph) \mu$, ensuring uniqueness (up to a positive multiplicative constant).
\end{proof}

\begin{remark}
    Note that,
    by the same reasoning as in the proof above, both the Dirichlet and Neumann realisation of $\form$
    are associated with a Green function (integral kernels of the associated operator on $L^2_\mu(\Graph)$) of class $L^\infty(\Graph\times \Graph)$. 
    Because by Corollary~\ref{cor:compact-resolvent}, both realisations have compact resolvent, the spectrum is equal on $L^p_\mu(\Graph)$ for each $1<p<\infty$. 
\end{remark}

\subsubsection{Leading eigenvalue asymptotics}

By Corollary~\ref{cor:compact-resolvent}, the operators $\OUDir$ and $\OUNeu$ have compact resolvent. Hence their spectra are purely discrete and consist of real eigenvalues of finite multiplicity accumulating only at $+\infty$.

\begin{proposition}
    \label{prop:linear-asymptotics-general}
    Enumerate the eigenvalues of $\OUDir$ and $\OUNeu$, repeated according to multiplicity, as
    \[
        0<\lambda_1^{\Dirichlet}\le \lambda_2^{\Dirichlet}\le \cdots \nearrow \infty
        \quad \text{and}\quad
        0=\lambda_1^{\Neumann}<\lambda_2^{\Neumann}\le \lambda_3^{\Neumann}\le \cdots \nearrow \infty,
    \]
    respectively.
    Then
    \[
        \lambda_n^{\Dirichlet}\sim \frac{2n}{\halflines},
        \qquad
        \lambda_n^{\Neumann}\sim \frac{2n}{\halflines},
        \qquad n\to\infty;
    \]
    where $\halflines=\cardinality{\EdgeSet_\infty}$ is the total number of half-lines attached to the compact part $\widetilde \Graph$.
    In particular,
    $
        \lambda_n^{\Dirichlet}=\calO(n),
    $
    and
    $
        \lambda_n^{\Neumann}=\calO(n).
    $
\end{proposition}

To describe eigenvalue asymptotics, we briefly use counting functions.
If $T$ is a self-adjoint operator with discrete spectrum, whose
eigenvalues are ordered increasingly and repeated according to
multiplicity, we denote by
\[
    N_T(\Lambda)
    \coloneqq
    \cardinality{\{n\in\bbN : \lambda_n(T)\le \Lambda\}}
\]
the number of eigenvalues (counted with multiplicity) not exceeding $\Lambda$.
We recall that $N_T$ is related to the eigenvalue sequence by
\[
    \lambda_n(T)\le \Lambda
    \quad\Longleftrightarrow\quad
    n \le N_T(\Lambda).
\]

\begin{proof}[Proof of Proposition~\ref{prop:linear-asymptotics-general}]
    We give the argument for $\OUDir$ and $\OUNeu$ simultaneously. Let
    $\widehat{\OU}^{\mathrm X}$, $\mathrm X\in\{\Dirichlet,\Neumann\}$, denote the
    operator obtained from $\OU^{\mathrm X}$ by imposing, in addition, Dirichlet
    conditions at all vertices of $\widetilde{\Graph}$ to which half-lines are
    attached. Thus the corresponding form domain is
    \[
        \widehat{\calV}^{\mathrm X}
        =
        \left\{
            f\in \domX{\form}:
            f(\vVertex)=0
            \text{ for every }\vVertex\in \LeafSet
        \right\}.
    \]
    This is a closed subspace of $\domX{\form}$ of finite codimension, bounded above
    by $\cardinality{\LeafSet}$.

    \emph{Step 1:
    We show that $\OU^{\mathrm X}$ and $\widehat{\OU}^{\mathrm X}$ have the same leading
    eigenvalue asymptotics}.
    Enumerate the eigenvalues of $\widehat{\OU}^{\mathrm X}$ (repeated according to multiplicity) as $\lambda_1(\widehat{\OU}^{\mathrm X})\le \lambda_2(\widehat{\OU}^{\mathrm X})\le \ldots$ and let $m$ be the codimension of
    $\widehat{\calV}^{\mathrm X}$ in $\domX{\form}$. Then
    \[
        \lambda_n^{\mathrm X}
        \le
        \lambda_n(\widehat{\OU}^{\mathrm X})
        \le
        \lambda_{n+m}^{\mathrm X}
        \quad 
        \text{for all }n\in\bbN.
    \]
    Indeed, the first inequality follows from the min--max
    principle because $\widehat{\calV}^{\mathrm X}\subset \domX{\form}$.
    To see the second inequality, let $R(\argument)$ denote the Rayleigh quotient corresponding to $(\form, \domX{\form})$. By the min--max principle,
    \[
        \lambda_n(\widehat{\OU}^{\mathrm X})
        =
        \inf_{\substack{M\subset \widehat{\calV}^{\mathrm X}\\ \dim M=n}}
        \sup_{0\neq f\in M}R(f),
    \qquad
    \text{and}
    \qquad
        \lambda_{n+m}^{\mathrm X}
        =
        \inf_{\substack{L\subset \domX{\form}\\ \dim L=n+m}}
        \sup_{0\neq f\in L}R(f).
    \]
    Fix an arbitrary subspace $L\subset\domX{\form}$ with $\dim L=n+m$
    and consider the quotient map
    \[
        \pi:\domX{\form}\to \domX{\form}/\widehat{\calV}^{\mathrm X}.
    \]
    Since $\widehat{\calV}^{\mathrm X}$ has co-dimension $m$, we have
    $
        \dim \pi(L)\le m.
    $
    By the rank--nullity theorem,
    \[
        \dim(L\cap\widehat{\calV}^{\mathrm X})
        =
        \dim L-\dim\pi(L)
        \ge
        (n+m)-m
        =
        n.
    \]
    As every $n$-dimensional subspace $M\subset L\cap\widehat{\calV}^{\mathrm X}$
    is admissible in the min--max formula for $\lambda_n(\widehat{\OU}^{\mathrm X})$,
    we obtain
    \[
        \lambda_n(\widehat{\OU}^{\mathrm X})
        \le
        \sup_{0\neq f\in M}R(f)
        \le
        \sup_{0\neq f\in L}R(f).
    \]
    Taking the infimum over all $(n+m)$-dimensional subspace $L\subset\domX{\form}$ yields
    \[
        \lambda_n(\widehat{\OU}^{\mathrm X})
        \le
        \inf_{\substack{L\subset \domX{\form}\\ \dim L=n+m}}
        \sup_{0\neq f\in L}R(f)
        =
        \lambda_{n+m}^{\mathrm X}.
    \]

    \emph{Step 2:
    We compute the eigenvalue asymptotics of $\widehat{\OU}^{\mathrm X}$}. 
    Note that the additional Dirichlet conditions at the attachment vertices decouple the compact
    core from the half-lines. For each $\vVertex\in\LeafSet$, denote by $\deg_\infty(\vVertex)$ the number of half-lines in $\EdgeSet_\infty$ attached to $\vVertex$.
    Then
    \[
        \widehat{\OU}^{\mathrm X}
        =
        B^{\mathrm X}
        \oplus
        \bigoplus_{\vVertex \in \LeafSet}
        \bigoplus_{j=1}^{\deg_\infty(\vVertex)}
        L_{\modulus{\vVertex}}^{\Dirichlet},
    \]
    where the compact part $B^{\Dirichlet}$ and $B^{\Neumann}$ are the
    operators acting as $\OU$ on each edge of $\widetilde{\Graph}$ with domains
    \[
        \left\{ f\in \widetilde{H_{\mu}^2}(\widetilde\Graph) : f(\vVertex)=0 \text{ for all }\vVertex\in\LeafSet \cup\{\oVertex\}\text{ and } f \text{ satisfies}~\eqref{eq:BC} \text{ for all } \vVertex\in\VertexSet \setminus (\{\oVertex\}\cup\LeafSet)\right\},
    \]
    and
    \[
        \left\{ f\in \widetilde{H_{\mu}^2}(\widetilde\Graph) : f(\vVertex)=0 \text{ for all }\vVertex\in\LeafSet\text{ and } f \text{ satisfies}~\eqref{eq:BC} \text{ for all } \vVertex\in\VertexSet \setminus \LeafSet\right\}
    \]
    respectively,
    and $L_{\modulus{\vVertex}  }^{\Dirichlet}$ is the Dirichlet
    Ornstein--Uhlenbeck operator on the half-line $\bigl[\modulus{\vVertex},\infty\bigr)$.

    The compact part is negligible in the leading asymptotics. Indeed, on the
    compact graph $\widetilde{\Graph}$ the Gaussian measure is comparable to
    Lebesgue measure, and the associated form is comparable with the standard
    metric-graph Laplacian form. Hence, it is known the eigenvalues of $B^{\mathrm X}$ grow
    quadratically, see \cite[Theorem~4.6]{Kur23},
    and in particular, its counting function satisfies
    \[
        N_{B^{\mathrm X}}(\Lambda)=\mathcal O(\sqrt{\Lambda}),
        \qquad
        \Lambda\to\infty.
    \]

    On the other hand, by Theorem~\ref{thm:OU-half-line-asymptotics}, each
    half-line Dirichlet Ornstein--Uhlenbeck operator satisfies
    \[
        \lambda_n(L_{\modulus{\vVertex}}^{\Dirichlet})\sim 2n,
        \qquad n\to\infty.
    \]
    So, by standard inversion between eigenvalue asymptotics and counting functions
    \[
        N_{L_{\modulus{\vVertex}}^{\Dirichlet}}(\Lambda)\sim \frac{\Lambda}{2},
        \qquad \Lambda\to\infty.
    \]
    Consequently, the counting function of the decoupled operator satisfies
    \[
        N_{\widehat{\OU}^{\mathrm X}}(\Lambda)
        =
        N_{B^{\mathrm X}}(\Lambda)
        +
        \sum_{\vVertex \in \LeafSet}\sum_{j=1}^{\deg_\infty(\vVertex)} N_{L_{\modulus{\vVertex}  }^{\Dirichlet}}(\Lambda)
        \sim
        \frac{\Lambda}2 \sum_{\vVertex \in \LeafSet} \deg_\infty(\vVertex) = \frac{\Lambda \halflines}2.
    \]
    Since the counting function is the asymptotic inverse of the ordered eigenvalue sequence, this gives
    \[
        \lambda_n(\widehat{\OU}^{\mathrm X})\sim \frac{2n}{\halflines},
        \qquad n\to\infty.
    \]

    \emph{Step 3}. 
    The index-shift comparison from Step~1 gives
    $
        \lambda_{n-m}(\widehat{\OU}^{\mathrm X})
        \le
        \lambda_n^{\mathrm X}
        \le
        \lambda_n(\widehat{\OU}^{\mathrm X})
    $
    for all $n>m$.
    Applying the asymptotic of $\lambda_n(\widehat{\OU}^{\mathrm X})$ from Step~2 on both sides,
    we obtain
    $
        \lambda_n^{\mathrm X}
        \sim
        \frac{2n}{\halflines}
    $.
\end{proof}

\begin{remarks}
    \leavevmode
    \begin{enumerate}[\upshape (a)]
        \item As a consequence of the linear growth of the eigenvalues in Proposition~\ref{prop:linear-asymptotics-general}, the
        resolvent of both realisations $\OUDir$ and $\OUNeu$ are of Hilbert--Schmidt
        class, and the associated semigroups are of trace class. 
        
        \item The proof of Proposition~\ref{prop:linear-asymptotics-general} shows that the
        leading eigenvalue asymptotics are determined entirely by the behaviour of
        the Ornstein--Uhlenbeck operator on the non-compact half-lines, with Dirichlet
        conditions at their endpoints. In contrast, the compact part
        $\widetilde{\Graph}$ contributes only lower-order terms, since its spectrum
        grows quadratically.
    
        In particular, the precise choice of vertex conditions on
        $\widetilde{\Graph}$ is largely irrelevant for the leading asymptotics, as
        long as the resulting realisation (with additional Dirichlet conditions at
        the attachment vertices) is self-adjoint. Such self-adjoint realisations can
        be characterised in terms of (possibly nonlocal) transmission conditions at
        the vertices; see, e.g., \cite[Theorem~1.4.4]{BerKuc13} or
        \cite[Section~6.5.1]{Mug14}.
    
        Consequently, the conclusion of
        Proposition~\ref{prop:linear-asymptotics-general} extends verbatim to more
        general vertex conditions on $\widetilde{\Graph}$, for instance to
        $\delta$-type conditions of the form
        \[
            \begin{aligned}
                f_{j}(\vVertex) = f_k(\vVertex)&
                \quad  &\text{for }\eEdge_j, \eEdge_k \sim \vVertex,\\
                \sum_{\eEdge_k\sim \vVertex}
                \frac{\partial f_k}{\partial n}(\vVertex)
                =
                \delta_{\vVertex} f(\vVertex)&
            \end{aligned}
        \]
        at all vertices $\vVertex\in\VertexSet$,
        for an arbitrary family of real parameters
        $(\delta_{\vVertex})_{\vVertex\in\VertexSet}$.
    \end{enumerate}
\end{remarks}

While Proposition~\ref{prop:linear-asymptotics-general} provides a coarse spectral estimate
valid for arbitrary rooted trees, it does not exploit any additional
geometric structure. In the case of regular rooted trees, a much more
refined analysis becomes possible, which we develop in Section~\ref{sec:regular-trees}.

\subsection{A Wentzell-type realisation}
    \label{sec:wentzell-realisation}

For completeness, we briefly discuss a further realisation of the Ornstein--Uhlenbeck operator on $\Graph$, obtained by coupling the dynamics on the edges with a dynamic boundary condition at the root $\oVertex$. In what follows, $\bbK$ denotes the scalar field.

Inspired by~\cite{MugnoloRomanelli2007} we consider the Hilbert space 
\[
	\calH\coloneqq L^2_\mu(\Graph)\times \bbK
\]
and identify $H^1_\mu(\Graph)$ in a canonical way with a closed subset of $\calH$:
\[
    H^1_\mu(\Graph)\equiv \calV\coloneqq \{(f,\psi)\in \calH : f\in H^1_\mu(\Graph)\text{ and }f(\oVertex)=\psi\}.
\]
Thus, $\calV$ encodes continuity at the root by identifying the trace $u(\oVertex)$ with the second component $\psi$.

We are thus in the setting described in \cite{MugnoloRomanelli2007}, \cite[Section~2]{Mugnolo2010} (up to the inessential detail that Gaussian -- rather than Lebesgue -- measure is endowing each half-line):~ we may so deduce from~\cite[Theorem~2.3]{Mugnolo2010} that the self-adjoint operator $\OUWen$  on $\calH$ formally associated with $(\form,\calV)$ is such that $-\OUWen$ generates an analytic, contractive, compact $C_0$-semigroup on $\calH$.

\begin{remark}
    At first, $\OUWen$ is only weakly defined as the operator associated with $(\form,\calV)$, but it can be shown as in Section~\ref{sec:dirichlet-neumann-realisation} (see also~\cite[Lemma~3.3]{MugnoloRomanelli2007}) that $\OUWen$ admits the following explicit description:
    \begin{align*}
    	\dom{\OUWen}
        &=
        \left\{
        \begin{pmatrix}
            f\\ \psi
        \end{pmatrix}
        \in \calV
        :
        f\in \widetilde{H^2_\mu}(\Graph) \text{ and } f \text{ satisfies}~\eqref{eq:BC} \text{ for all } \vVertex\in\VertexSet \setminus \{\oVertex\}
    \right\},\\
    	\OUWen \begin{pmatrix}
            f\\ \psi
        \end{pmatrix}
        &=
        \begin{pmatrix}
            \OU f\\[0.3em]
            \displaystyle-\frac{1}{\sqrt \pi}\sum_{\eEdge\sim \oVertex}
        \frac{\partial f_{\eEdge}}{\partial n}(\oVertex)
        \end{pmatrix}
    \end{align*}
    The coefficient in the second component is the boundary coefficient produced by the weighted integration by parts. Indeed, at the root all incident edges are parametrised away from $\oVertex$, and therefore the root contribution is
    \[
        -\frac12 G(\oVertex)
        \sum_{\eEdge\sim\oVertex}
        \frac{\partial f_{\eEdge}}{\partial n}(\oVertex)
        \overline{h(\oVertex)}.
    \]
    Since $G(\oVertex)=G(0)=2/\sqrt\pi$, this coefficient equals
    $-G(\oVertex)/2=-1/\sqrt\pi$. In the Dirichlet case this term
    vanishes because the test functions vanish at $\oVertex$, while in
    the Neumann case it yields only a homogeneous Kirchhoff condition and
    the non-zero scalar factor can be divided out. In the Wentzell
    realisation, however, the boundary value is part of the Hilbert space
    $L^2_\mu(\Graph)\times\bbK$, so the boundary contribution becomes the
    $\bbK$-component of the operator and the coefficient remains.
\end{remark}

It can be proved precisely as \cite[Theorem~3.4]{Mugnolo2010} that the semigroup generated by $-\OUWen$ leaves invariant the order interval
\[
    [\zero_{\calH},\one_{\calH}]=\left\{(f,\psi)\in \mathcal H:0\le f(x)\le 1\text{ for a.e. }x\in \Graph\text{ and }0\le \psi\le 1\right\}.
\]
Finally, observe that $\one_{\calH}\in \dom{\OUWen}$ and that $\OUWen \one_{\calH}=0$. In particular, $e^{-t\OUWen} \one_{\calH} =\one_{\calH}$ for all $t>0$.
We can finally summarize our findings and state the following.

\begin{theorem}
	The operator $-\OUWen$ generates an analytic, contractive, compact,  Markovian $C_0$-semigroup on the Hilbert space $\calH$.
\end{theorem}

A direct computation shows that this semigroup delivers the solution for a parabolic problem with dynamic transmission condition at the root
\[
    \left\{
    \begin{aligned}
        \partial_t u(t,x)
        &= \frac{1}{2} u''(t,x) - \modulus{x} u'(t,x),
        && x\in \Graph\setminus \VertexSet,\\
        \partial_t \psi(t)
        &= \frac{1}{\sqrt \pi}\sum_{\eEdge\sim \oVertex}
        \frac{\partial u_{\eEdge}}{\partial n}(t,\oVertex),\\
        u(t,\oVertex) &= \psi(t),\\
         \sum_{\eEdge\sim\vVertex}
        \frac{\partial u_{\eEdge}}{\partial n}(t,\vVertex)
        &=0,
        && \vVertex\in\VertexSet\setminus\{\oVertex\},\\
        u(0,x)&=u_0(x),\quad \psi(0)=\psi_0.
    \end{aligned}
    \right.
\]
where the second equation is well defined because, by analyticity of the
semigroup, for each $t>0$ the solution belongs to $\dom{\OUWen}$. In
particular, $u(t,\argument)\in \widetilde H^2_\mu(\Graph)$, so the normal
derivatives at the root are well defined. Moreover, since $(u(t,\argument),\psi(t))\in\calV$, we have $u(t,\oVertex)=\psi(t)$.
More precisely, for all $t\ge 0$,
\[
    \begin{pmatrix}
        u(t,\argument)\\
        \psi(t)
    \end{pmatrix}
    =
    e^{-t\OUWen}
    \begin{pmatrix}
        u_0\\
        \psi_0
    \end{pmatrix},
    \qquad t\ge 0.
\]

Finally, the arguments used in Theorem~\ref{thm:invariant-measure} extend to this setting, yielding existence and uniqueness (up to normalisation) of an invariant measure for the associated semigroup.

\section{Ornstein--Uhlenbeck operator on regular rooted trees}
    \label{sec:regular-trees}

In the rest of the paper, we shall restrict our studies to a more special class of rooted trees, called \emph{regular rooted trees} (see the definition below), whose radial symmetry allows for a reduction of the Ornstein--Uhlenbeck operator to a family of one-dimensional problems and consequently obtain finer spectral results.

\subsection{Preliminaries on regular rooted trees}

Recall that, in Section~\ref{sec:rooted-trees}, our rooted tree $\Graph$ is obtained from a compact rooted metric tree $\widetilde{\Graph}$ by attaching half-lines to each vertex in the boundary set $\LeafSet$.
We now impose a regularity condition on this construction.

\begin{definition}
    The rooted tree $\Graph$ is called \emph{regular} if there exist $N_{\Graph}\in\bbN$ and sequences
    \[
        (t_k)_{k=0}^{N_{\Graph}} \subset [0,\infty),
        \qquad
        (b_k)_{k=0}^{N_{\Graph}} \subset \bbN,
    \]
    such that the following hold:
    \begin{enumerate}[\upshape (i)]
        \item
        $
            0=t_0<t_1<\ldots<t_{N_{\Graph}};
        $
    
        \item the vertex set admits a partition into generations:
        \[
            \VertexSet = \bigsqcup_{k=0}^{N_{\Graph}} \VertexSet_k
        \quad \text{where} \quad
            \VertexSet_k = \{ \vVertex\in \VertexSet : \modulus{\vVertex} = t_k \};
        \]
    
        \item for each $k\in\{0,\ldots,N_{\Graph}\}$, every vertex $\vVertex\in \VertexSet_k$ is incident to exactly $b_k$ outgoing edges (children);
    
        \item the vertex set $\VertexSet_{N_{\Graph}}$ coincides with the set of leaves $\LeafSet$ of $\widetilde{\Graph}$, therefore each such vertex is incident to $b_{N_{\Graph}}$ half-lines.
    \end{enumerate}
    The numbers $t_0,\ldots,t_{N_{\Graph}}$ are called the \emph{generation levels} and $b_0,\ldots,b_{N_{\Graph}}$ the \emph{branching numbers} of the tree.
\end{definition}

\begin{remarks}
    \leavevmode
    \begin{enumerate}[\upshape (a)]
        \item By the metric structure introduced in Section~\ref{sec:rooted-trees}, each edge
        $\eEdge = \edge{\vVertex}{\wVertex}$ with
        $\vVertex \in \VertexSet_k,\wVertex \in \VertexSet_{k+1}$, and $\vVertex \preceq \wVertex$ is identified with the interval $[t_k,t_{k+1}]$.
        In particular,
        $
            d(\vVertex,\wVertex) = t_{k+1}-t_k
        $.

        \item The case $b_{N_{\Graph}}=1$ corresponds to a tree that does not branch beyond the last generation, while $b_{N_{\Graph}}\ge 2$ yields genuine branching at infinity. The construction rules out vertices with no outgoing compact edge
        before the last generation, in turn $b_k\ge 1$ for $k=0,\ldots,N_{\Graph}-1$.
        
        \item The number of half-lines attached to $\widetilde \Graph$ is given by
        $
            m_\infty = b_0 b_1 b_2\ldots b_{N_{\Graph}}
        $.
    \end{enumerate}
\end{remarks}

For the analysis that follows, it is convenient to associate subtrees to vertices and edges. For vertices $\vVertex,\wVertex \in \VertexSet$ and edge $\eEdge=\edge{\vVertex}{\wVertex}$ where $\vVertex\preceq \wVertex$, we define
\[
    \Graph_{\vVertex} \coloneqq \{ x\in\Graph : x \succeq \vVertex \}
    \quad \text{and} \quad
    \Graph_{\eEdge} \coloneqq \eEdge \cup \Graph_{\wVertex}.
\]
In particular, $\Graph_{\oVertex}=\Graph$.
More generally, any subtree $\calS \subset \Graph$ inherits a metric structure from $\Graph$. We associate to $\calS$ its \emph{branching function}
\[
    g_{\calS} : [0,\infty) \to \bbN_0,
    \qquad
    g_{\calS}(t) \coloneqq \cardinality{ \{ x\in\calS : \modulus{x} = t\}}.
\]

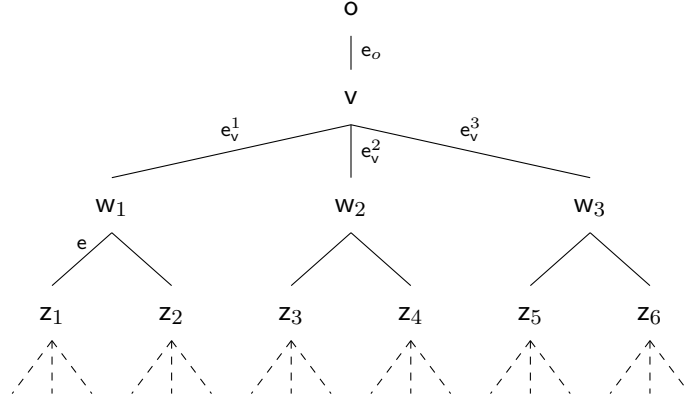
\begin{figure}
    \centering
    \begin{forest}
    [
        $\oVertex$
	[
            $\vVertex$, edge label={node[midway,right,font=\scriptsize]{$\eEdge_o$}},for tree={l sep=7mm}, 
	    [
                $\wVertex_1$, edge label={node[midway,above,font=\scriptsize]{$\eEdge_{\vVertex}^1$}}
		    [
                    $\zVertex_1$, edge label={node[midway,above,font=\scriptsize]{$\eEdge$}}
                    [
                        ,edge=dashed
                    ]
                    [
                        ,edge=dashed
                    ]
                    [
                        ,edge=dashed
                    ]
			]
                [
                    $\zVertex_2$
                    [
                        ,edge=dashed
                    ]
                    [
                        ,edge=dashed
                    ]
                    [
                        ,edge=dashed
                    ]
                ]
            ]
            [
                $\wVertex_2$, edge label={node[midway,right,font=\scriptsize]{$\eEdge_{\vVertex}^2$}}
                [
                    $\zVertex_3$
                    [
                        ,edge=dashed
                    ]
                    [
                        ,edge=dashed
                    ]
                    [
                        ,edge=dashed
                    ]
                ]
                [
                    $\zVertex_4$
                    [
                        ,edge=dashed
                    ]
                    [
                        ,edge=dashed
                    ]
                    [
                        ,edge=dashed
                    ]
                ]
            ]
            [
                $\wVertex_3$, edge label={node[midway,above,font=\scriptsize]{$\eEdge_{\vVertex}^3$}}
                [
                    $\zVertex_5$
                    [
                        ,edge=dashed
                    ]
                    [
                        ,edge=dashed
                    ]
                    [
                        ,edge=dashed
                    ]
			]
                [
                    $\zVertex_6$
                    [
                        ,edge=dashed
                    ]
                    [
                        ,edge=dashed
                    ]
                    [
                        ,edge=dashed
                    ]
                ]
            ]
        ]	
    ]
    \end{forest}
    \caption{A regular rooted tree. The dashed edges indicate the
    attached half-lines at the boundary vertices.}
    \label{fig:illustration}
\end{figure}

\begin{example}
    \label{ex:regular-rooted-tree}
    In Figure~\ref{fig:illustration}, we have a regular rooted tree $\Graph=(\VertexSet, \EdgeSet)$ with
    \[
        \VertexSet = \{ \oVertex, \vVertex, \wVertex_1, \wVertex_2, \wVertex_3, \zVertex_1,\ldots,\zVertex_6\}
        \quad\text{and}\quad
        \EdgeSet \supseteq \{ \eEdge_\oVertex, \eEdge_{\vVertex}^1, \eEdge_{\vVertex}^2, \eEdge_{\vVertex}^3, \eEdge\}. 
    \]
    Further, $N_{\Graph}=3$ and the branching sequence is given by $b_0=1, b_1=3, b_2=2$, and $b_3=3$.
    Thus the root has one child, the next generation branches into three
    vertices, each of those branches into two vertices, and finally each
    boundary vertex is incident to three attached half-lines.

    The generation sizes are
    \[
        \cardinality{\VertexSet_0}=1,\qquad
        \cardinality{\VertexSet_1}=b_0=1,\qquad
        \cardinality{\VertexSet_2}=b_0b_1=3,\qquad
        \cardinality{\VertexSet_3}=b_0b_1b_2=6.
    \]
    Hence $\LeafSet=\VertexSet_3$ consists of six vertices, and since each of
    them carries $b_3=3$ half-lines, the total number of half-lines is
    $
        \halflines= b_0b_1b_2b_3=18
    $.
    If we write
    \[
        m_k\coloneqq \cardinality{\VertexSet_k}=b_0b_1\ldots b_{k-1},
        \qquad k=1,2,3,
    \]
    then $m_1=1, m_2=3$, and $m_3=6$.
    These numbers will later appear as multiplicities in the orthogonal
    decomposition of the operator reduction.

    The branching function of the whole tree is
    \[
        g_{\Graph}(t)=
        \begin{cases}
            1,      \quad& t\in[0,t_1),\\
            3,      \quad& t\in[t_1,t_2),\\
            6,      \quad& t\in[t_2,t_3),\\
            18,     \quad& t\in[t_3,\infty)
        \end{cases}
    \]
    and if $\eEdge$ is the first edge emanating from
    $\wVertex_1$, then the subtree
    $\Graph_{\eEdge}$ has branching function
    \[
        g_{\Graph_{\eEdge}}(t)=
        \begin{cases}
            0,      \quad& t\in[0,t_2),\\
            1,      \quad& t\in[t_2,t_3),\\
            3,      \quad& t\in[t_3,\infty).
        \end{cases}
    \]
\end{example}

\subsection{\texorpdfstring{Orthogonal decomposition of $L^2_\mu(\Graph)$}{Orthogonal decomposition of L^2_μ}}

Following \cite{Solomyak2004}, we now exploit the radial symmetry encoded by the sequences $(t_k)$ and $(b_k)$ of a regular rooted tree to decompose the space $L^2_\mu(\Graph)$ into orthogonal subspaces that are invariant under the Ornstein--Uhlenbeck operator. 

Let $\Graph$ be a regular rooted tree. For any subtree $\calS\subset\Graph$, we consider the space of all \emph{symmetric} functions,
\[
    \calF_{\calS}
    \coloneqq
    \left\{
        f\in L^2_\mu(\Graph):
        \begin{array}{l}
            f(x)=0 \text{ for } x\notin\calS,\\[0.2em]
            f(x)=f(y)\text{ whenever }x,y\in\calS\text{ and }\modulus{x}=\modulus{y}
        \end{array}
    \right\},
\]
i.e., $\calF_{\calS}$ consists of all functions that are supported in $\calS$ and depend there only on the distance from the root. Clearly $\calF_{\calS}$ is a closed subspace of $L^2_{\mu}(\Graph)$ and the corresponding orthogonal projection is given by radial averaging on $\calS$, namely
\[
    (P_{\calS}f)(x)
    \coloneqq
    \begin{cases}
        g_{\calS}(\modulus{x})^{-1}
        \displaystyle\sum_{\substack{y\in\calS\\\modulus{y}=\modulus{x}}} f(y), \qquad
        &x\in\calS,\\[1em]
        0,\qquad &x\notin\calS.
    \end{cases}
\]

We next introduce a weighted measure on $[0,\infty)$. Since the density of $\mu$ depends only on the distance from the root, we define
\begin{equation}\label{eq:measure-nu}
    \nu(\dxShort t) \coloneqq  
        g_{\Graph}(t)G(t)\dx t
        =
        g_{\Graph}(t)\frac{2}{\sqrt\pi}e^{-t^2}\dx t.
\end{equation}
In this way, the radial part of $L^2_\mu(\Graph)$ is identified with the weighted space $L^2_\nu([0,\infty))$.

Next, fix $k\in \{0,1,\ldots,N_{\Graph}\}$ and $\vVertex\in\VertexSet_k$. If $\eEdge\in\EdgeSet$ is an edge with initial vertex $\vVertex$, then every function in $\calF_{\eEdge}\coloneqq \calF_{\Graph_{\eEdge}}$ is radial, and the map
\[
    J_{\eEdge}:\calF_{\eEdge}\to L^2_\nu([t_k,\infty)),
    \qquad
    J_{\eEdge}f(t):=(b_0\cdots b_k)^{-1/2}f(x),\quad \modulus{x}=t,
\]
is an isometric isomorphism. 
Indeed, for $t\ge t_k$ the number of points in the subtree
$\Graph_{\eEdge}$ at distance $t$ from the root is
$(b_0\cdots b_k)^{-1}g_{\Graph}(t)$. Hence, if $f\in\calF_{\eEdge}$ and $\widehat f$ denotes its radial representative,
then
\[
    \int_{\Graph_{\eEdge}} \modulus{f(x)}^2\mu(\dxShort x)
    =
    \int_{t_k}^\infty
    \modulus{(b_0\cdots b_k)^{-1/2}\widehat f(t)}^2 \nu(\dxShort t).
\]
Therefore, denoting by
\[
    \eEdge_{\vVertex}^1,\ldots,\eEdge_{\vVertex}^{b_k}
\]
all the edges emanating from $\vVertex$, we obtain that the orthogonal sum
\[
    \widetilde{\calF}_{\vVertex}
    \coloneqq
    \bigoplus_{r=1}^{b_k}\calF_{\eEdge_{\vVertex}^r}
\]
is isometrically isomorphic to $\left(L^2_\nu([t_k,\infty))\right)^{b_k}$ via
\[
    J_{\vVertex}:
    \widetilde{\calF}_{\vVertex}
    \to
    \left(L^2_\nu([t_k,\infty))\right)^{b_k},
    \qquad
    (f_1,\ldots,f_{b_k})
    \mapsto
    \left(J_{\eEdge_{\vVertex}^1}f_1,\ldots,J_{\eEdge_{\vVertex}^{b_k}}f_{b_k}\right).
\]
To separate the symmetric and oscillatory branch modes, we diagonalise the finite-dimensional branching structure by means of the discrete Fourier basis of $\bbC^{b_k}$. Let
$
    \omega_k:=e^{2\pi i/b_k},
$
and for $j=1,\ldots,b_k$ define
\[
    \mathbf h_k^{\langle j\rangle}
    \coloneqq
    b_k^{-1/2}
    \left(\omega_k^j,\omega_k^{2j},\ldots,\omega_k^{(b_k-1)j},1\right)
    \quad \text{and}\quad
    \calF_{\vVertex}^{\langle j\rangle}
    \coloneqq
    J_{\vVertex}^{-1}
    \left\{
        \mathbf h_k^{\langle j\rangle}\varphi:
        \varphi\in L^2_\nu([t_k,\infty))
    \right\},
\]
i.e., $\calF_{\vVertex}^{\langle j\rangle}$ is the subspace corresponding to the $j$-th branching mode at the vertex $\vVertex$.
The mode $j=b_k$ is the symmetric one, while the modes
$
    j=1,\ldots,b_k-1
$
describe oscillations between the outgoing branches. In particular,
$
    \calF_{\Graph}=\calF_{\oVertex}^{\langle b_0\rangle}
$
coincides with the global radial subspace.
We point out that
\begin{equation}
    \label{eq:solomyak-construction-non-radial}
    J_{\vVertex}f=\mathbf h_k^{\langle j\rangle}\varphi
    \Leftrightarrow
    f(x)
    =
    \sqrt{b_0\cdots b_k}\, \omega_k^{rj}\,\varphi(\modulus{x}),
    \quad \text{for all }
    x\in \Graph_{\eEdge_{\vVertex}^r}, \,
    r=1,\ldots,b_k.
\end{equation}
We also write
\[
    J_{\oVertex}:\calF_{\Graph}\to L^2_\nu([0,\infty)),
    \qquad
    J_{\oVertex}f(t)=f(x),\quad \modulus{x}=t,
\]
for the global radial identification.

Finally, for each $f\in L_{\mu}^2(\Graph)$, we define its radial component $f_{\oVertex} \in L^2_{\nu}([0, \infty))$ and the branching components $f_{\vVertex} ^{\langle j \rangle} \in L^2_{\nu}([t_k, \infty))$ as follows:
\begin{align}
    \label{eq:decomposition-projection-onto-components}
    f_{\oVertex}              &\coloneqq J_{\oVertex} P_{\oVertex} f\\
    f_{\vVertex} ^{\langle j \rangle} &\coloneqq b_k^{-1/2} \sum_{r=1}^{b_k} \omega_k^{-rj} J_{\eEdge_{\vVertex} ^r} P_{\eEdge_{\vVertex}^r} f \qquad j = 1,\ldots, b_k;
\end{align}
here we have used the shorthands $P_{\oVertex} \coloneqq P_{\Graph_{\oVertex}}$ and $P_{\eEdge}\coloneqq P_{\Graph_{\eEdge}}$ for an edge $\eEdge \in \EdgeSet$. We point out that the symmetric mode at the root coincides with the global radial component, i.e.,
$
f_{\oVertex} = f_{\oVertex}^{\langle b_0\rangle}.
$

We are now ready to state the promised decomposition of $L^2_\mu(\Graph)$. We use the convention that direct sums over empty index sets are omitted.

\begin{proposition}
    \label{prop:decomposition-solomyak}
    For a regular rooted tree $\Graph$, one has the decomposition
    \[
        L^2_\mu(\Graph)
        =
        \calF_{\Graph}
        \oplus
        \bigoplus_{k=0}^{N_{\Graph}}
        \ \bigoplus_{\vVertex\in\VertexSet_k}
        \ \bigoplus_{j=1}^{b_k-1}
        \calF_{\vVertex}^{\langle j\rangle}.
    \]
    and the identity
    \begin{equation}
        \label{eq:decomposition-solomyak-identity}
        \int_{\Graph} \modulus{f(x)}^2 \mu(\dxShort x) 
        =
        \int_0^\infty \modulus{f_{\oVertex}(t)}^2 \nu(\dxShort t) + \sum_{k= 0}^{N_{\Graph}} 
        \ \sum_{\vVertex\in\VertexSet_k}
        \ \sum_{j=1}^{b_{k}-1} 
        \int_{t_k}^{\infty} \modulus{f_{\vVertex} ^{\langle j \rangle}(t)}^2 \nu(\dxShort t)
    \end{equation}
    holds for every $f\in L_{\mu}^2(\Graph)$. 
    Additionally, the decomposition is compatible with the Sobolev structure. More precisely, if $f\in H_{0,\mu}^{1}(\Graph)$, then $f_{\oVertex} \in H_{0,\nu}^{1}([0, \infty))$ and $f_{\vVertex}^{\langle j \rangle} \in H_{0,\nu}^{1}([t_k, \infty))$ with
    \begin{equation}
        \label{eq:decomposition-solomyak-identity-derivative}
        \int_{\Graph} \modulus{f'(x)}^2 \mu(\dxShort x) 
        = 
        \int_0^\infty \modulus{\frac{\dxShort f_{\oVertex}(t)}{\dxShort t}}^2 \nu(\dxShort t)+
        \sum_{k= 0}^{N_{\Graph}}
        \ \sum_{\vVertex\in\VertexSet_k}
        \ \sum_{j=1}^{b_{k}-1}
        \int_{t_k}^{\infty} \modulus{\frac{\dxShort f_{\vVertex} ^{\langle j \rangle}(t)}{\dxShort t}}^2 \nu(\dxShort t).
    \end{equation}
\end{proposition}

\begin{proof}
    Since the radial density $G(t)= \frac{2}{\sqrt \pi}e^{-t^2}$ 
    depends only on the distance from the root and $\Graph$ is regular, the decomposition of $L^2_\mu(\Graph)$ follows from \cite[Theorem~2.3]{NaimarkSolomyak2001} applied at each vertex. In particular, the case $b_0>1$ is handled by performing the decomposition at the root in the same way as at higher generations.

    To see the compatibility with the Sobolev structure, let $f\in H_{0,\mu}^{1}(\Graph)$. 
    We know from \cite[Theorems~2.3 and~3.1]{NaimarkSolomyak2001} that the derivative function leaves the spaces $\calF_{\Graph}$ and $\calF_{\vVertex} ^{\langle j \rangle}$ invariant. This means that $f_{\oVertex}'\in \calF_{\Graph}$ and $\left(f_{\vVertex}^{\langle j \rangle}\right)'\in \calF_{\vVertex}^{\langle j \rangle}$. Further, we even obtain from \cite[Theorems~2.3 and 3.1]{NaimarkSolomyak2001}, that the subspaces
    \[
        H_{0,\Graph}
        \coloneqq \calF_{\Graph} \cap H^1_{0,\mu}(\Graph) 
        \quad \text{and} \quad 
        H_{0,\vVertex}^{\langle j\rangle}
        \coloneqq \calF_{\vVertex}^{\langle j\rangle} \cap H^1_{0,\mu}(\Graph)
    \]
    are orthogonal in $H_{0,\mu}^{1}(\Graph)$. Indeed, passing from length measure to the
    Gaussian measure only multiplies the level-wise inner product by the common radial
    factor $G(t)$, and therefore the orthogonality of the branching modes is preserved.

    Moreover, due to~\eqref{eq:decomposition-projection-onto-components}, the projections act locally on each branch, are compatible with the decomposition of $f$ and preserve the boundary condition (since $f(\oVertex)=0$ and projections vanish at the root). Thus, $f_{\oVertex} \in H_{0,\nu}^{1}([0, \infty))$, $f_{\vVertex} ^{\langle j \rangle} \in H_{0,\nu}^{1}([t_k, \infty))$, and~\eqref{eq:decomposition-solomyak-identity-derivative} holds.
\end{proof}

\begin{remark}
    The argument in the proof of Proposition~\ref{prop:decomposition-solomyak} also shows that the subspaces
    \[
        \calH_{\Graph}
        \coloneqq \calF_{\Graph} \cap H^1_{\mu}(\Graph) 
        \quad \text{and} \quad 
        \calH_{\vVertex}^{\langle j\rangle}
        \coloneqq \calF_{\vVertex}^{\langle j\rangle} \cap H^1_{\mu}(\Graph)
    \]
    orthogonally decompose $H_{\mu}^{1}(\Graph)$.
\end{remark}

\subsection{Reducing the Ornstein--Uhlenbeck operator}

We now identify the action of the Ornstein--Uhlenbeck operator on the components arising in Proposition~\ref{prop:decomposition-solomyak}.

Let $\Graph$ be a regular rooted tree and fix $k\in \{0, \ldots, N_{\Graph}\}$. On $L^2_{\nu}([t_k, \infty))$, we consider the sesquilinear form $\form_k(\argument,\argument)$ given by 
\[
    \form_k(\varphi, \psi) = \frac12 \int_{t_k}^\infty \varphi'(t)\overline{\psi'(t)} \nu(\dxShort t).
\]
The corresponding Dirichlet and Neumann forms are given by restricting $\form_k$ to the domains
\[
    \domDir{\form_k} \coloneqq H^1_{0,\nu}( [t_k,\infty) )
    \quad\text{and}\quad
    \domNeu{\form_k} \coloneqq H^1_{\nu}( [t_k,\infty) )
\]
respectively. Both  $(\form_k, \domDir{\form_k})$ and $(\form_k, \domNeu{\form_k})$ are closed, densely defined, symmetric, and non-negative forms on $L^2_{\nu}([t_k, \infty))$. Denote the corresponding self-adjoint operators by $\OUDirLevel{k}$ and $\OUNeuLevel{k}$ respectively.

The domains of the operators $\OUDirLevel{k}$ and $\OUNeuLevel{k}$ are described in Proposition~\ref{prop:domain-reduced-operator} below. First, let us state the main reduction result of this subsection. In the sequel, $A^{[r]}$ stands for the orthogonal sum of $r$ copies of a self-adjoint operator $A$, with the convention that $A^{[0]}$ is omitted.

\begin{theorem}
    \label{thm:operator-decomposition}
    The operator $\OU^{\Dirichlet}$ is unitarily equivalent to
    \[
        \OUDirLevel{0}
        \oplus
        \bigoplus_{k=0}^{N_{\Graph}}
        \left(\OUDirLevel{k}\right)^{[m_k(b_k-1)]},
    \]
    whereas the operator $\OU^{\Neumann}$ is unitarily equivalent to
    \[
        \OUNeuLevel{0}
        \oplus
        \bigoplus_{k=0}^{N_{\Graph}}
        \left(\OUDirLevel{k}\right)^{[m_k(b_k-1)]}.
    \]
    Here $m_0\coloneqq 1$ and $m_k\coloneqq b_0b_1\cdots b_{k-1}$ for $k\ge 1$. 
\end{theorem}

\begin{proof}
    By Proposition~\ref{prop:decomposition-solomyak}, the decomposition of
    $L^2_\mu(\Graph)$ is orthogonal and is compatible with the energy form.
    Hence it remains to identify the form domains induced on the radial and
    non-radial components.

    First consider the radial component. For a general function $f\in L^2_\mu(\Graph)$,
    its radial representative is
    $
        f_{\oVertex}=J_{\oVertex}P_{\oVertex}f
    $.
    Therefore, if $f\in\calF_{\Graph}$ and $u=J_{\oVertex}f$, then
    $
        f_{\oVertex}(0)=f(\oVertex)
    $.
    Thus the Dirichlet graph form induces the condition $f_{\oVertex}(0)=0$ on the
    radial component, while the Neumann graph form imposes no vanishing condition
    at $0$. Consequently, the radial component is governed by $\OUDirLevel{0}$ in the
    Dirichlet case and by $\OUNeuLevel{0}$ in the Neumann case.

    We now consider a non-radial component. Fix $\vVertex\in\VertexSet_k$,
    $j\in\{1,\ldots,b_k-1\}$, and let
    $
        \eEdge_{\vVertex}^1,\ldots,\eEdge_{\vVertex}^{b_k}
    $
    be the edges emanating from $\vVertex$. If $f\in \calF_{\vVertex}^{\langle j\rangle}$, by definition, there exists $\varphi\in L^2_{\nu}([t_k,\infty))$ such that
    \[
        J_{\vVertex}f=\mathbf h_k^{\langle j\rangle}\varphi,
        \qquad
        \mathbf h_k^{\langle j\rangle}
        =
        b_k^{-1/2}
        (\omega_k^j,\omega_k^{2j},\ldots,\omega_k^{(b_k-1)j},1).
    \]
    By~\eqref{eq:solomyak-construction-non-radial}, along the edge $\eEdge_{\vVertex}^r$ one has $f(x)=C_k\,\omega_k^{rj}\,\varphi(\modulus{x})$ with $C_k= \sqrt{b_0\ldots b_k}$. Passing to the limit
    $\modulus{x}\downarrow t_k$ yields the values $C_k \omega_k^{rj} \varphi(t_k)$ along the
    different edges. Since functions in $H^1_\mu(\Graph)$ are continuous at
    $\vVertex$, these values must coincide for all $r$, hence
    \[
        \omega_k^{rj}\varphi(t_k)=\omega_k^{sj}\varphi(t_k)
        \qquad \text{for all } r,s.
    \]
    As the vector $\mathbf h_k^{\langle j\rangle}$ is not constant for
    $j\in\{1,\ldots,b_k-1\}$, it follows that $\varphi(t_k)=0$.
    Thus the one-dimensional representative of every non-radial component lies in $H^1_{0,\nu}([t_k,\infty))$.
    Consequently, every non-radial component, both for the Dirichlet and for
    the Neumann graph realisation, is governed by the reduced Dirichlet operator
    $\OUDirLevel{k}$.

    It remains to count multiplicities. At the root, there is one radial mode
    and $b_0-1$ non-radial modes. Hence the Dirichlet realisation gives
    $b_0$ copies of $\OUDirLevel{0}$, whereas the Neumann realisation gives one
    copy of $\OUNeuLevel{0}$ and $b_0-1$ copies of $\OUDirLevel{0}$.
    For $k\ge 1$, there are $m_k=b_0b_1\cdots b_{k-1}$
    vertices in $\VertexSet_k$, and each contributes $b_k-1$ non-radial modes.
    This yields the asserted decompositions.
\end{proof}

\begin{corollary}
    \label{cor:semigroup-decomposition}
    For $\mathrm X\in\{\Dirichlet,\Neumann\}$, the semigroup
    $(e^{-t\OU^{\mathrm X}})_{t\ge 0}$ generated by $-\OU^{\mathrm X}$ on $L^2_\mu(\Graph)$ is unitarily equivalent to the orthogonal direct-sum semigroup
    \[
        e^{-t\OUXLevel{0}}
        \oplus
        \bigoplus_{k=0}^{N_{\Graph}}
        \left(e^{-t\OUDirLevel{k}}\right)^{[m_k(b_k-1)]};
    \]
    where $m_k=b_0b_1\ldots b_{k-1}$ for $k\ge 1$ and $m_0=1$.
\end{corollary}

\begin{proof}
    By Theorem~\ref{thm:operator-decomposition}, the operator $\OU^{\mathrm X}$ is unitarily equivalent to the orthogonal sum
    \[
        \OUXLevel{0}
        \oplus
        \bigoplus_{k=0}^{N_{\Graph}}
        \left(\OUDirLevel{k}\right)^{[m_k(b_k-1)]}.
    \]
    Since each operator $\OUXLevel{k}$ is self-adjoint and non-negative, it generates a strongly continuous contraction semigroup $(e^{-t\OUXLevel{k}})_{t\ge 0}$ on $L^2_\nu([t_k,\infty))$. 

    By standard functional calculus for self-adjoint operators (see, e.g.,~\cite[Section~A-I.3.8]{Nag86}), the semigroup generated by the orthogonal sum is the corresponding orthogonal product of the semigroups generated by the summands. The claim follows.
\end{proof}

Let us now characterize the domain of the operators $\OUDirLevel{k}$ and $\OUNeuLevel{k}$.

\begin{proposition}
    \label{prop:domain-reduced-operator}
    Let $\mathrm X\in \{\Dirichlet, \Neumann\}$.
    For $k< N_{\Graph}$, a function $\varphi\in L^2_{\nu}([t_k, \infty))$ lies in $\dom{\OUXLevel{k}}$ if and only if each of the following conditions is satisfied:
    \begin{enumerate}[\upshape (a)]
        \item \emph{(Local regularity)}
        \label{prop:domain-reduced-operator:itm:local-integrability}
        For each $j$ with $k<j\le N_{\Graph}$, the function $\varphi\restrict{[t_{j-1},t_j]}$ lies in  $H^2_{\nu}([t_{j-1},t_j])$
        and, on the terminal half-line, $\varphi\restrict{[t_{N_{\Graph}},\infty)} \in H^2_{\nu}([t_{N_{\Graph}},\infty))$.
    
        \item \emph{(Continuity and transmission)}
        \label{prop:domain-reduced-operator:itm:continuity-and-transmission}
        The function $\varphi$ is continuous on $[t_k,\infty)$. Furthermore, for each $j$ with $k<j\le N_{\Graph}$ the jump conditions $\varphi'(t_j+)=b_j^{-1}\varphi'(t_j-)$ hold.
        
        \item \emph{(Boundary condition at $t_k$)}
        \label{prop:domain-reduced-operator:itm:boundary-condition}
        At point $t_k$, the function $\varphi$ satisfies
        \[
            \begin{cases}
                \varphi(t_k)=0, \quad& \text{if } \mathrm X=\Dirichlet,\\
                \varphi'(t_k+)=0, \quad& \text{if } \mathrm X=\Neumann.
            \end{cases}
        \]
    \end{enumerate}

    In addition,
    \begin{align*}
        \dom{\OUDirLevel{ N_{\Graph} } }
        & = H_{0,\nu}^1([t_{N_{\Graph}}, \infty)) \cap H^2_{\nu}([t_{N_{\Graph}},\infty))
        \text{ and}
        \\
        \dom{\OUNeuLevel{ N_{\Graph} } }
        & = \left\{ \varphi \in H^2_{\nu}([t_{N_{\Graph}},\infty)) : \varphi'(t_{N_{\Graph}}+)=0  \right\}.
    \end{align*}
\end{proposition}

Theorem~\ref{thm:operator-decomposition} and Proposition~\ref{prop:domain-reduced-operator} together tell us that the Ornstein--Uhlenbeck operator on a regular rooted tree is reduced to a finite orthogonal sum of one-dimensional Ornstein--Uhlenbeck operators on half-lines, coupled through non-standard transmission conditions at the levels $t_k$.

For the proof of Proposition~\ref{prop:domain-reduced-operator}, we need the following:

\begin{lemma}
    \label{lem:reduced-H2-estimate}
    Consider the notations introduced above.
    \begin{enumerate}[\upshape (a)]
        \item 
        \label{lem:reduced-H2-estimate:itm:sufficient}
        If $\varphi\in H^2_\nu([t_{N_{\Graph}},\infty))$, then
        $
            t\varphi'\in L^2_\nu([t_{N_{\Graph}},\infty))
        $.
    
        \item
        \label{lem:reduced-H2-estimate:itm:necessary}
        Let $k\le N_{\Graph}$ and let $\varphi\in H^1_\nu([t_k,\infty))$ be such that 
        $
            \psi(t)\coloneqq -\frac12 \varphi''(t)+\modulus{t}\varphi'(t) 
        $
        lies in $L^2_\nu(I)$, in the sense of distributions,
        on each interval 
        $
            I\in \left\{[t_{j-1},t_j]: k<j\le N_{\Graph}\right\}\cup\left\{[t_{N_{\Graph}},\infty)\right\}
        $. 
        Then
        $(\varphi\restrict{I})''\in L^2_\nu(I)$ on each such interval. 
        % Consequently, 
        % \[
        %     \int_I \left(\modulus{\varphi''(t)}^2+\modulus{\varphi'(t)}^2+\modulus{\varphi(t)}^2\right)\nu(\dxShort t)<\infty
        % \] 
        % for every such interval $I$.
    \end{enumerate}
\end{lemma}

\begin{proof}
    For convenience, we write $T\coloneqq t_{N_{\Graph}}$. Since the branching function $g_{\Graph}$ is piecewise constant, there exists $c>0$ such that
    the reduced measure has the form
    \[
        \nu(\mathrm dt)= c e^{-t^2}\dxShort t\eqqcolon w(t)\dxShort t,
        \quad \text{on}\quad
        [T,\infty).
    \]

    (a) By assumption, $\varphi', \varphi'' \in L^2_{\nu}([T,\infty))$, and in turn, $\varphi'\in H^1_{\mathrm{loc}}([T,\infty))$. Hence, using 
    $                       \left(\modulus{\varphi'}^2\right)'=2\re(\varphi''\overline{\varphi'})
    $
    and $(tw(t))'= (1-2t^2)w(t)$,
    the following integration by parts is
    justified for each $R>T$:
    \begin{equation}
        \label{eq:integration-by-parts-h2-estimate}
        \begin{aligned}
            \int_T^R t\re(\varphi''\overline{\varphi'})(t) w(t)\dx t
            &=
            \frac12\int_T^R tw(t) \left(\modulus{\varphi'(t)}^2\right)' \dx t\\
            &=
            \frac12\left[t\modulus{\varphi'(t)}^2w(t) \right]_T^R
            -
            \frac12\int_T^R (tw(t) )'\modulus{\varphi'(t)}^2\dx t\\
            &=
            \frac12\left[t\modulus{\varphi'(t)}^2w(t) \right]_T^R
            -
            \frac12\int_T^R \modulus{\varphi'(t)}^2w(t)\dx t
            +
            \int_T^R t^2\modulus{\varphi'(t)}^2w(t)\dx t.
        \end{aligned}
    \end{equation}
    Equivalently, 
    \begin{align*}
        \int_T^R t^2\modulus{\varphi'(t)}^2w(t)\dx t
        &=
        \int_T^R t\re(\varphi''\overline{\varphi'})(t) w(t)\dx t
        -
        \frac12\left[t\modulus{\varphi'(t)}^2w(t) \right]_T^R
        +
        \frac12\int_T^R \modulus{\varphi'(t)}^2w(t)\dx t\\
        &\le
        \int_T^R t\modulus{\varphi''\varphi'}(t) w(t)\dx t
        +
        \frac12T\modulus{\varphi'(T+)}^2w(T)
        +
        \frac12\int_T^R \modulus{\varphi'(t)}^2w(t)\dx t.
    \end{align*}
    By Young's inequality $ab \le 2^{-1}(a^2+b^2)$, applied pointwise with $a=\modulus{\varphi''(t)}$ and $b=t\modulus{\varphi'(t)}$,
    we therefore obtain,
    \[
        \int_T^R t^2\modulus{\varphi'(t)}^2w(t)\dx t
        \le
        \int_T^R \modulus{\varphi''(t)}^2w(t)\dx t
        +
        T\modulus{\varphi'(T+)}^2w(T)
        +
        \int_T^R \modulus{\varphi'(t)}^2w(t)\dx t.
    \]
    Recalling $\varphi', \varphi'' \in L^2_{\nu}([T,\infty))$ and letting $R\to \infty$, it follows that
    $t\varphi'\in L^2_\nu([T,\infty))$.

    (b) Let $I$ be one of the intervals from the assumption. 
     To keep the notation light, we write again
    $\varphi$ and $\psi$ for the restrictions of these functions to $I$;
    all derivatives and Sobolev spaces in the next paragraphs are understood
    on $I$.
    
    Since
    $
        \varphi''(t)=-2\psi(t)+2\modulus{t}\varphi'(t)
    $
    on $I, \varphi'\in L^2_\nu(I)$, and $\modulus{t}$ is bounded on each compact interval $[t_{j-1},t_j]$, it follows immediately that $\varphi''\in L^2_\nu([t_{j-1},t_j])$.

    It remains to treat the terminal half-line
    $
        [T,\infty)=[t_{N_{\Graph}},\infty)
    $.
    First, observe that $\varphi'\in H^1_{\mathrm{loc}}([T,\infty))$. Indeed, on
    every compact interval $[T,R]$, the function $t$ is bounded, and
    $
        \varphi''=-2\psi+2t\varphi' \in L^2([T,R]),
    $
    as a distribution.
    Hence, the integration by parts in~\eqref{eq:integration-by-parts-h2-estimate} is
    justified on $[T,R]$. Coupling it with
    \begin{align*}
        \modulus{\psi(t)}^2 
        =
        \modulus{-\frac12 \varphi''+t\varphi'}^2 
        =
        \frac14 \modulus{\varphi''(t)}^2
        +
        t^2\modulus{\varphi'(t)}^2
        -
        t\re(\varphi''\overline{\varphi'} )(t),
    \end{align*}
    we obtain
    \[
        \int_T^R \modulus{\psi(t)}^2 \nu(\dxShort t) 
        =
        \frac14\int_T^R \modulus{\varphi''(t)}^2 w(t)\dx t
        +
        \frac12\int_T^R \modulus{\varphi'(t)}^2w(t)\dx t
        -
        \frac12\left[t\modulus{\varphi'(t)}^2w(t) \right]_T^R.
    \]
    Since $\varphi'\in L^2_\nu([T,\infty))$, we have
    $
        \liminf_{R\to\infty} R \modulus{\varphi'(R)}^2 w(R) = 0
    $.
    As a result, 
    \[
        \frac14\int_T^\infty \modulus{\varphi''(t)}^2 w(t)\dx t
        +
        \frac12\int_T^\infty \modulus{\varphi'(t)}^2w(t)\dx t
        +
        \frac12T\modulus{\varphi'(T+)}^2w(T) 
        \le 
        \int_T^\infty \modulus{\psi}^2 \nu(\dxShort t) <\infty
    \]
    In particular, 
    $
        \varphi''\in L^2_\nu([T,\infty))
    $.
\end{proof}

\begin{proof}[Proof of Proposition~\ref{prop:domain-reduced-operator}]
    The characterisation for $k=N_{\Graph}$ follows from Lemma~\ref{lem:reduced-H2-estimate} together with the definition of the operators $\OUDirLevel{N_{\Graph}}$ and $\OUNeuLevel{N_{\Graph}}$, since in that case there are no interior transmission points. So, fix $k<N_{\Graph}$. 
    
    \emph{Necessity}. 
    Since $\OUXLevel{k}$ is the operator associated with the closed form $(\form_k, \domX{\form_k})$, each 
    $\varphi \in \dom{\OUXLevel{k}}$ satisfies $\varphi\in \domX{\form_k}$ 
    and there exists $\psi \in L^2_{\nu}([t_{k},\infty))$ such that
    \begin{equation}
        \label{eq:weak-reduced-X}
        \form_k(\varphi, \phi) 
        = \duality{\psi}{\phi}_{L^2_{\nu}([t_{k},\infty))} 
        \qquad \text{for all }\phi \in \domX{\form_k}.
    \end{equation}

    \ref{prop:domain-reduced-operator:itm:local-integrability} 
    Choose test functions $\phi$ compactly supported inside a single 
    interval $(t_{j-1}, t_j)$ with $k<j\le N_{\Graph}$ or inside the half-line 
    $(t_{N_{\Graph}}, \infty)$. For such $\phi$, the boundary terms at the 
    internal nodes vanish and integration by parts~\eqref{eq:integration-by-parts-gaussian}
    yields distributionally that 
    \[
      -\tfrac{1}{2} \varphi''(t) + \modulus{t}\,\varphi'(t) = \psi(t) \in L^2_{\nu}
    \]
    on each interval $(t_{j-1},t_j)$ and on the half-line $(t_{N_{\Graph}}, \infty)$. Thus the hypothesis of Lemma~\ref{lem:reduced-H2-estimate}\ref{lem:reduced-H2-estimate:itm:necessary} is satisfied.
    So, $\varphi''\in L^2_\nu(I)$ on every interval $I$
    under consideration. Since $\varphi\in\domX{\form_k}\subset
    H^1_\nu([t_k,\infty))$, the $L^2_\nu$-integrability of
    $\varphi$ and $\varphi'$ is already known. So, $\varphi$ has the asserted piecewise $H^2_\nu$-regularity.

    \ref{prop:domain-reduced-operator:itm:continuity-and-transmission}
    Next, we test \eqref{eq:weak-reduced-X} with functions $\phi\in \domX{\form_k}$ supported in a small neighbourhood of some transmission point $t_j$, where $k<j\le N_{\Graph}$. Once again, integrating by parts, on the adjacent intervals and collecting the boundary terms, we obtain
    \begin{equation}
        \label{eq:bdry-terms}
        \begin{split}
            2\duality{\psi}{\phi}_{L^2_{\nu}([t_{k},\infty))} 
              = \sum_{i=k+1}^{N_{\Graph}} 
                 & g_{\Graph}(t_i-) \Big( \varphi'(t_i-)-b_i\varphi'(t_i+)\Big) 
                 \phi(t_i)G(t_i)  \\
              &- g_\Graph(t_k^+)\,\varphi'(t_k^+)\,\phi(t_k)G(t_k)
                 +2 \duality{\psi}{\phi}_{L^2_{\nu}([t_{k},\infty))}.
        \end{split}
    \end{equation}
    From this, the interior jump conditions follow from the vanishing 
    of the coefficients of $\phi(t_j)$; note that $\phi(t_i)=0$ whenever $i\ne j$.
    Continuity on $[t_k, \infty)$ holds since the form domain contains the continuity condition at each level.

    \ref{prop:domain-reduced-operator:itm:boundary-condition}
    The condition at $t_k$ depends on $\mathrm X$. If $\mathrm X=\Dirichlet$, then $\varphi\in H^1_{0,\nu}([t_k,\infty))$, so $\varphi(t_k)=0$. If $\mathrm X=\Neumann$, then~\eqref{eq:bdry-terms} with $\phi$ supported near $t_k$ forces
    \[
        g_\Graph(t_k^+)\,\varphi'(t_k^+) = 0 \text{ or equivalently }\varphi'(t_k^+) = 0.
    \]

    \emph{Sufficiency}. 
    Conversely, let $\varphi\in L^2_{\nu}([t_k, \infty))$ satisfy conditions~\ref{prop:domain-reduced-operator:itm:local-integrability}--\ref{prop:domain-reduced-operator:itm:boundary-condition}.
    By the assumed piecewise $H^2_\nu$-regularity
    and the continuity condition, we have
    $
        \varphi\in H^1_\nu([t_k,\infty))
    $.
    Moreover, in the Dirichlet case, condition~\ref{prop:domain-reduced-operator:itm:boundary-condition} ensures
    $\varphi(t_k)=0$,
    $
        \varphi\in H^1_{0,\nu}([t_k,\infty))
    $. In either case, $\varphi \in \domX{\form_k}$.
    Define
    \[
        \psi(t):= -\tfrac{1}{2} \varphi''(t) + \modulus{t}\,\varphi'(t),
    \]
    pointwise on each open interval. Local integrability condition and Lemma~\ref{lem:reduced-H2-estimate}\ref{lem:reduced-H2-estimate:itm:sufficient} together ensure that $\psi\in L^2_\nu([t_k,\infty))$. Let $\phi \in \domX{\form_k}$ with compact support.
    Integrating by parts, using the interior transmission conditions, and the boundary contribution at $t_k$, one can check that~\eqref{eq:weak-reduced-X} holds for all 
    $\phi \in \domX{\form_k}$ with compact support, and in turn for all $\phi \in \domX{\form_k}$. 
    As a result, $\varphi \in\dom{\OUXLevel{k}}$.
\end{proof}

\subsection{Spectral results}

We now exploit the operator decomposition obtained in Theorem~\ref{thm:operator-decomposition} to derive spectral properties of the Ornstein--Uhlenbeck operator on regular rooted trees. All spectral information reduces to that of the one-dimensional operators $\OUXLevel{k}$, combined with the multiplicities induced by the branching structure. 

We first establish a monotonicity property for the
reduced operators, which reflects the fact that moving the boundary point
further away enlarges the domain. We then combine this with a comparison
principle for weighted Gaussian measures to obtain lower bounds on the
eigenvalues of each reduced operator. Finally, we lift these one-dimensional
estimates to the full operator using the orthogonal decomposition and counting
arguments based on eigenvalue multiplicities.

Let $\Graph$ be a regular rooted tree and let $\OUDirLevel{k}$ and $\OUNeuLevel{k}$ respectively denote the reduced operator from the operator decomposition of $\OUDir$ and $\OUNeu$ from Theorem~\ref{thm:operator-decomposition}.
In what follows, for each $k\in \{0,\ldots,N_{\Graph}\}$ we denote the eigenvalues of $\OUDirLevel{k}$ and $\OUNeuLevel{k}$ (ordered increasingly) by
\[
    \lambda_n^{\Dirichlet, (k)}
    \quad \text{and}\quad
    \lambda_n^{\Neumann, (k)},
    \qquad (n\in \bbN)
\]
respectively.

The monotonicity property (Proposition~\ref{prop:monotonicity-reduced-spectrum} below) is a consequence of the following general domain monotonicity principle for
one-dimensional Dirichlet forms on half-lines:

\begin{lemma}
    \label{lem:half-line-domain-monotonicity}
    Let $0\le a<b$ and
    let $\eta$ be a positive Radon measure with full support on $[a,\infty)$ that is absolutely continuous with respect to Lebesgue measure.
    For $c\in\{a,b\}$, consider the form
    \[
        \mathfrak q_c(u)=\frac12\int_c^\infty \modulus{u'(t)}^2\,\eta(\dxShort t)
    \]
 with domain
    \[
        \text{either}\quad \domDir{\mathfrak q_c}\coloneqq H^1_{0,\eta}([c,\infty))
        \quad\text{or}\quad
        \domNeu{\mathfrak q_c}\coloneqq H^1_{\eta}([c,\infty)).
    \]
    Assume that these forms are closed and that, for each
    $c\in\{a,b\}$, the embedding
    $
        H^1_\eta([c,\infty))
        \hookrightarrow
        L^2_\eta([c,\infty))
    $
    is compact. Then the associated self-adjoint operators have compact
    resolvent.
     Denote the eigenvalues corresponding to $(\mathfrak q_c, \domDir{\mathfrak q_c})$ and $(\mathfrak q_c, \domNeu{\mathfrak q_c})$ (ordered increasingly) by 
    $\big(\lambda_n^{\Dirichlet}(c)\big)_{n\in\bbN}$ and $\big(\lambda_n^{\Neumann}(c)\big)_{n\in\bbN}$, respectively. Then
    \[
        \lambda_n^{\Dirichlet}(a)\le \lambda_n^{\Dirichlet}(b)
        \quad\text{and}\quad
        \lambda_n^{\Neumann}(a)\le \lambda_n^{\Neumann}(b)
        \quad \text{for all }n\in\bbN.
    \]
    
    In particular, if $\eta$ is the Gaussian measure centered at $0$, then $\lambda_n^{\Dirichlet}(a)\ge 2n-1$ and $\lambda_n^{\Neumann}(a)\ge 2(n-1)$ for all $n\in\bbN$.
\end{lemma}

\begin{proof}
    Let $a<b$ and consider the extension operator
    \[
        E : H^1_{\eta}([b,\infty)) \to H^1_{\eta}([a,\infty)),
        \qquad
        (E u)(t) \coloneqq 
        \begin{cases}
            u(b), & t\in [a,b),\\
            u(t), & t\ge b.
        \end{cases}
    \]
    Then $E$ is linear and injective, and for all $u\in H^1_{\eta}([b,\infty))$ one has
    \[
        \mathfrak q_a(E u,E u)=\mathfrak q_b(u,u),
        \quad\text{and}\quad
        \norm{E u}_{L^2_\eta([a,\infty))}^2 \ge \norm{u}_{L^2_\eta([b,\infty))}^2.
    \]
    Consequently, the Rayleigh quotients
    \[
        R_a(u) \coloneqq \frac{\mathfrak q_a(u,u)}{\norm{u}_{L^2_\eta([a,\infty))}^2},
        \qquad u \neq 0,
    \]
    satisfy
    \begin{equation}
        \label{eq:rayleigh-quotient-abstract}
        R_a(E u) \le R_b(u)
        \quad \text{for all } 0\neq u\in H^1_{\eta}([b,\infty)).
    \end{equation}

    Next, let $\calS_n(a)$ denote the family of all $n$-dimensional subspaces of $H^1_\eta([a,\infty))$ and set
    \[
        \calT_n(a) := \left\{ E (L) : L \in \calS_n(b) \right\} \subseteq \calS_n(a).
    \]
    By the min-max principle,
    \begin{align*}
        \lambda_n^{\Neumann}(a)
        & = \inf_{M \in \calS_n(a)} \sup_{0\ne u\in M} R_a (u) 
         \le \inf_{M \in \calT_n(a)} \sup_{0\ne u\in M} R_a (u) \\
        & = \inf_{L \in \calS_n(b)} \sup_{0\ne u\in E(L)} R_a (u)\qquad (\text{by definition of } \calT_n(a)) \\
        & = \inf_{L \in \calS_n(b)} \sup_{0\ne u\in  L} R_a (E u) \qquad (\text{using injectivity of } E)\\
        & \le \inf_{L \in \calS_n(b)} \sup_{0\ne u\in  L} R_b (u)
        = \lambda_n^{\Neumann}(b),
    \end{align*}
    where we have used $\calT_n(a)\subseteq \calS_n(a)$ and~\eqref{eq:rayleigh-quotient-abstract}, respectively, for the two inequalities.

    The Dirichlet case is proved analogously, replacing $Eu$ on $[a,b)$ by $0$, in which case the Rayleigh quotient is preserved exactly.

    The last assertion follows by recalling that $\lambda_n^{\Dirichlet}(0)=2n-1$ and $\lambda_n^{\Neumann}(0)=2(n-1)$ for all $n\in \bbN$, see e.g., \cite[Lemma~4.7]{MugnoloRhandi2022}.
\end{proof}

\begin{proposition}
    \label{prop:monotonicity-reduced-spectrum}
    For each $n\in \bbN$, we have
    \[
        \lambda_n^{\Dirichlet, (0)} \le \lambda_n^{\Dirichlet, (1)} \le \ldots \le \lambda_n^{\Dirichlet, (N_{\Graph})}
        \quad\text{and}\quad
        \lambda_n^{\Neumann, (0)} \le \lambda_n^{\Neumann, (1)} \le \ldots \le \lambda_n^{\Neumann, (N_{\Graph})}.
    \]
\end{proposition}

\begin{proof}
    Fixing $k\in \{0,\ldots,N_{\Graph}-1\}$, the assertion follows by applying Lemma~\ref{lem:half-line-domain-monotonicity} with $\eta=\nu, a=t_k$, and $b=t_{k+1}$.
\end{proof}

\begin{remark}
    \label{rem:rayleigh-comparison}
    Fix $a\ge0$ and let $\eta_g$ be the measure on $[a,\infty)$
    given by
    \[
        \eta_g(\dxShort t)
        =
        g(t)G(t)\dx t
        =
        g(t)\frac{2}{\sqrt\pi}e^{-t^2}\dx t,
    \]
    where $g$ is a measurable function satisfying
    $
        0< c \le g(t) \le C< \infty
    $
    for all $t\ge a$.
    Then for every $\varphi\in H^1_{\eta_g}([a,\infty))$,
    \[
        \frac{\int_a^\infty \modulus{\varphi'(t)}^2\eta_g(\dxShort t)}
             {\int_a^\infty \modulus{\varphi(t)}^2\eta_g(\dxShort t)}
        \ge
        \frac{c}{C}
        \frac{\int_a^\infty \modulus{\varphi'(t)}^2G(t)\dx t}
             {\int_a^\infty \modulus{\varphi(t)}^2G(t)\dx t}.
    \]
    In particular, by the min-max principle, the eigenvalues of the
    operator associated with the form
    \[
        \frac12\int_a^\infty \modulus{\varphi'(t)}^2\eta_g(\dxShort t)
    \]
    (with either Dirichlet or Neumann boundary condition at $t=a$) are
    bounded from below by $c\,C^{-1}$ times the corresponding
    eigenvalues for the one-dimensional Gaussian measure $G(t)\dx t$.
\end{remark}

We now combine the operator decomposition with the previous comparison
principles to obtain global spectral bounds for $\OUDir$ and $\OUNeu$.

\begin{proposition}
    \label{prop:spectrum-refined}
    Enumerate the eigenvalues of $\OUDir$ and $\OUNeu$, repeated according to multiplicity, as
    \[
        0<\lambda_1^{\Dirichlet}\le \lambda_2^{\Dirichlet}\le \ldots 
        \quad \text{and}\quad
        0=\lambda_1^{\Neumann}<\lambda_2^{\Neumann}\le \lambda_3^{\Neumann}\le \ldots
    \]
    respectively. For each $q\in \bbN_0$, we have
    \[
        \lambda_{(q+1)\halflines}^{\Dirichlet} \ge \ldots \ge \lambda_{q\,\halflines+1}^{\Dirichlet} \ge \frac{(2q+1)b_0}{\halflines}
    \]
    as well as
    \[
        \lambda_{q\,\halflines+1}^{\Neumann} \ge \frac{2q\cdot b_0}{\halflines}
        \quad\text{and}\quad
        \lambda_{(q+1)\halflines}^{\Neumann} \ge \ldots \ge \lambda_{q\,\halflines+2}^{\Neumann} \ge \frac{(2q+1)b_0}{\halflines};
    \]
    where $\halflines= b_0b_1\cdots b_{N_{\Graph}}$.
\end{proposition}

We point out that the second Neumann block inequality in Proposition~\ref{prop:spectrum-refined} is void when $\halflines=1$.

To exploit the operator decomposition at the level of eigenvalue counts,
it is convenient to work with counting functions, which were used in a
heuristic way in Section~\ref{sec:ou-general}.
Let $T$ be a self-adjoint operator with discrete spectrum. For lower bounds on eigenvalues, we use the function
$N_T(\Lambda-)$ that denotes the number of eigenvalues of $T$ (counted with multiplicity) that are
strictly less than $\Lambda$. For upper bounds,  we use $N_T(\Lambda)$ that denotes the number of eigenvalues of $T$ (counted with multiplicity) that are
less than or equal to $\Lambda$. These two quantities differ only at eigenvalues:
\[
    N_T(\Lambda-) \le N_T(\Lambda) \le N_T(\Lambda-) + m(\Lambda);
\]
where $m(\Lambda)$ denotes the multiplicity of $\Lambda$ as an eigenvalue of $T$.
Observe that
\begin{equation}
    \label{eq:counting-function-direct-sum}
    N_{\bigoplus_{j=1}^r T_j}(\Lambda-) = \sum_{j=1}^r N_{T_j}(\Lambda-)
    \quad\text{and}\quad
    N_{\bigoplus_{j=1}^r T_j}(\Lambda) = \sum_{j=1}^r N_{T_j}(\Lambda).
\end{equation}
Also, if $\lambda_n(T)$ denotes the $n$-th eigenvalue of $T$, then
\begin{align*}
     N_T(\Lambda) \ge n\,
     \Longleftrightarrow\,
    \lambda_n(T) \le \Lambda\,
    \quad\text{and}\quad
    \lambda_n(T) \ge \Lambda\
    \Longleftrightarrow\,
    N_T(\Lambda-) \le n-1.
\end{align*}
We now apply these properties to the orthogonal decomposition from
Theorem~\ref{thm:operator-decomposition}.

\begin{proof}[Proof of Proposition~\ref{prop:spectrum-refined}]
    We use the operator decomposition from Theorem~\ref{thm:operator-decomposition}.
    
    \emph{Dirichlet case.}
    As the eigenvalues are listed in increasing order, it suffices to show $\lambda_{q\,\halflines+1}^{\Dirichlet} \ge  (2q+1)b_0\halflines^{-1}$.
    The operator $\OUDir$ is the orthogonal sum of
    \[
        b_0+\sum_{k=1}^{N_{\Graph}} m_k\ (b_k-1) = \halflines
    \]
    reduced Dirichlet
    operators; where $m_k \coloneqq b_0b_1\ldots b_{k-1}$. Fix $q\in\bbN_0$.
    Since $b_0\le g_{\Graph}\le \halflines$, Lemma~\ref{lem:half-line-domain-monotonicity} combined with Remark~\ref{rem:rayleigh-comparison} yield for each reduced operator that
    \[
        \lambda_{q+1}^{\Dirichlet,(k)} \ge \frac{(2q+1)b_0}{\halflines},
    \]
    i.e., each reduced Dirichlet operator has at most $q$ eigenvalues strictly below $(2q+1)b_0\halflines^{-1}$.
    In particular,
    \[
        N_{\OUDirLevel{k}}\big((2q+1)b_0\halflines^{-1}-\big) \le q.
    \]
    Summing over the $\halflines$ components, we obtain
    $
        N_{\OUDir}\big((2q+1)b_0\halflines^{-1}-\big)
        \le q\,\halflines
    $,
    or equivalently, $\lambda_{q\,\halflines+1}^{\Dirichlet} \ge  (2q+1)b_0\halflines^{-1}$.
    
    \emph{Neumann case.}
    The operator $\OUNeu$ decomposes into one reduced Neumann operator and
    $\halflines-1$ reduced Dirichlet operators.
    Again by Lemma~\ref{lem:half-line-domain-monotonicity} and Remark~\ref{rem:rayleigh-comparison}, we have
    \[
        \lambda_{q+1}^{\Neumann,(0)} \ge \frac{2q\cdot b_0}{\halflines}
        \quad \text{and}\quad
        \lambda_{q+1}^{\Dirichlet,(k)} \ge \frac{(2q+1)b_0}{\halflines}.
    \]
    It follows that
    \[
        N_{\OUNeuLevel{0}}\big(2qb_0\halflines^{-1}-\big)
        \le q
        \quad \text{and}\quad
        N_{\OUDirLevel{k}}\big(2qb_0\halflines^{-1}-\big)
        \le q.
    \]
    Hence
    \[
        N_{\OUNeu}\big(2qb_0\halflines^{-1}-\big)
        \le q + (\halflines-1)q = q\,\halflines,
    \]
    which is equivalent to, $\lambda_{q\,\halflines+1}^{\Neumann} \ge 2qb_0\halflines^{-1}$.
    Finally, we are left to show that $\lambda_{q\,\halflines+2}^{\Neumann} \ge (2q+1)b_0\halflines^{-1}$. Indeed, as above, we can infer from Lemma~\ref{lem:half-line-domain-monotonicity} and Remark~\ref{rem:rayleigh-comparison} that
    \[
        N_{\OUNeuLevel{0}}\big((2q+1)b_0\halflines^{-1}-\big)
        \le q+1,
        \quad \text{and}\quad
        N_{\OUDirLevel{k}}\big((2q+1)b_0\halflines^{-1}-\big)
        \le q.
    \]
    Thus,
    $
        N_{\OUNeu}\big((2q+1)b_0\halflines^{-1}-\big)
        \le (q+1)+(\halflines-1)q = q\,\halflines+1
    $, 
    which yields
    $
        \lambda_{q\,\halflines+2}^{\Neumann} \ge (2q+1)b_0\halflines^{-1}
    $.
\end{proof}

We next localise the spectrum of the full operator within blocks
determined by the reduced spectra.

\begin{proposition}
    \label{prop:spectrum-block-localization}
    For each $n\in\bbN$, the Dirichlet eigenvalues satisfy
    \[
        \lambda_n^{\Dirichlet,(0)}
        \le
        \lambda_{\halflines(n-1)+1}^{\Dirichlet}
        \le \cdots \le
        \lambda_{\halflines n}^{\Dirichlet}
        \le
        \lambda_n^{\Dirichlet,(N_{\Graph})}
    \]
    and
    the Neumann eigenvalues satisfy
    \[
        \lambda_n^{\Neumann,(0)}
        \le
        \lambda_{\halflines(n-1)+1}^{\Neumann}
        \le \cdots \le
        \lambda_{\halflines n}^{\Neumann}
        \le
        \lambda_n^{\Dirichlet,(N_{\Graph})}.
    \]
\end{proposition}

\begin{proof}
    We write
    \[
        m_0 \coloneqq 1
        \quad\text{and}\quad
        m_k\coloneqq b_0b_1b_2\ldots b_{k-1}, \quad k\ge 1
    \]
    so that $\halflines=b_0+\sum_{k\ge 1} m_k(b_k-1)$. To obtain the bounds, we use  the counting functions $N_T(\Lambda)$ and $N_T(\Lambda-)$ introduced previously.

    \emph{Dirichlet case}. 
    As the eigenvalues are ordered increasingly, it suffices to show the bounds 
    $\lambda_n^{\Dirichlet,(0)}
    \le
    \lambda_{\halflines(n-1)+1}^{\Dirichlet}$
    and 
    $\lambda_{\halflines n}^{\Dirichlet}
    \le
    \lambda_n^{\Dirichlet,(N_{\Graph})}$.
    By Theorem~\ref{thm:operator-decomposition},
    $\OUDir$ is unitarily equivalent to the orthogonal sum of $\halflines$ reduced
    Dirichlet operators, counted with multiplicity. 
    From Proposition~\ref{prop:monotonicity-reduced-spectrum},
    $
        \lambda_n^{\Dirichlet,(k)}
        \le
        \lambda_n^{\Dirichlet,(N_{\Graph})}
    $
    for every
    $k=0,\ldots,N_{\Graph}$,
    so the $n$-th eigenvalue of each reduced Dirichlet component is at most
    $\lambda_n^{\Dirichlet,(N_{\Graph})}$.
    Hence $n\le N_{\OUDirLevel{k}}\left(\lambda_n^{\Dirichlet,(N_{\Graph})}\right)$.
    Consequently, using~\eqref{eq:counting-function-direct-sum}
    \begin{align*}
        N_{\OUDir}\left(\lambda_n^{\Dirichlet,(N_{\Graph})}\right)
        &= b_0 N_{\OUDirLevel{0}}\left(\lambda_n^{\Dirichlet,(N_{\Graph})}\right) +
        \sum_{k\ge1 } m_k(b_k-1) N_{\OUDirLevel{k}}\left(\lambda_n^{\Dirichlet,(N_{\Graph})}\right) \\
        & \ge  \left(b_0+\sum_k m_k(b_k-1)\right) \cdot n = \halflines n.
    \end{align*}
    In other words, $\OUDir$ has at least $\halflines n$ eigenvalues $\le \lambda_n^{\Dirichlet,(N_{\Graph})}$, which guarantees $\lambda_{\halflines n}^{\Dirichlet} \le \lambda_n^{\Dirichlet,(N_{\Graph})}$.

    For the other bound, recall from Proposition~\ref{prop:monotonicity-reduced-spectrum} that $\lambda_n^{\Dirichlet, (0)}\le \lambda_n^{\Dirichlet, (k)}$ for each $k$. 
    Thus, there are at most $n-1$ eigenvalues of $\OUDirLevel{k}$ strictly less than $\lambda_n^{\Dirichlet, (0)}$, equivalently, $N_{\OUDirLevel{k}}(\lambda_n^{\Dirichlet, (0)}-)\le n-1$. We infer, as above, that 
    $
        N_{\OUDir}\left(\lambda_n^{\Dirichlet,(0)}-\right)
        \le
        \halflines(n-1)
    $,
    which is equivalent to the desired bound
    $
        \lambda_{\halflines(n-1)+1}^{\Dirichlet}
        \ge
        \lambda_n^{\Dirichlet,(0)}.
    $

    \emph{Neumann case}. 
    The Neumann case is treated similarly, with the additional observation that the decomposition contains exactly one Neumann component and $\halflines-1$ Dirichlet components.

    For the upper bound, observe that for each $k$, Proposition~\ref{prop:monotonicity-reduced-spectrum} implies
    \begin{equation}
        \label{eq:spectrum-block-localization:monotonicity}
        \lambda_n^{\Neumann,(0)}
        \le
        \lambda_n^{\Neumann,(k)}
        \le
        \lambda_n^{\Dirichlet,(k)}
        \le
        \lambda_n^{\Dirichlet,(N_{\Graph})};
    \end{equation}
    where the second inequality holds by the inclusion of the corresponding form domains.
    Therefore every component in the Neumann decomposition has at least $n$
    eigenvalues not exceeding $\lambda_n^{\Dirichlet,(N_{\Graph})}$. Since there
    are $\halflines$ components in total,
    $
        N_{\OUNeu}\left(\lambda_n^{\Dirichlet,(N_{\Graph})}\right)
        \ge
        \halflines n,
    $
    equivalently
    $
        \lambda_{\halflines n}^{\Neumann}
        \le
        \lambda_n^{\Dirichlet,(N_{\Graph})}
    $.

    For the lower bound,~\eqref{eq:spectrum-block-localization:monotonicity} ensures that the reduced Neumann component and each reduced Dirichlet component
    has at most $n-1$ eigenvalues strictly smaller than
    $\lambda_n^{\Neumann,(0)}$. Therefore
    \[
        N_{\OUNeu}\left(\lambda_n^{\Neumann,(0)}-\right)
        \le
        \halflines(n-1)
    \]
    which yields
    $
        \lambda_{\halflines(n-1)+1}^{\Neumann}
        \ge
        \lambda_n^{\Neumann,(0)}
    $.
\end{proof}

We are now ready to state the main result of this section, which
shows that the spectrum of $\OU^{\mathrm X}$ grows linearly with $n$,
with effective scaling factor $\halflines^{-1}$ at the lower end and
harmonic-oscillator asymptotics at the upper end.

\begin{theorem}
    \label{thm:regular-tree-spectral-blocks}
    Let $\mathrm X\in\{\Dirichlet,\Neumann\}$. For each $n\in\bbN$, the eigenvalues of $\OU^{\mathrm X}$ satisfy
    \[
        \lambda_n^{\mathrm X,(0)}
        \le
        \lambda_{\halflines(n-1)+1}^{\mathrm X}
        \le \cdots \le
        \lambda_{\halflines n}^{\mathrm X}
        \le
        \lambda_n^{\Dirichlet,(N_{\Graph})},
    \]
    where the upper bound obeys
    \[
        \lambda_n^{\Dirichlet,(N_{\Graph})}
        \sim
        2n - 1 + \frac{4 t_{N_{\Graph}}}{\pi}\sqrt{n}
        + \frac{4t_{N_{\Graph}}^2}{\pi^2},
        \qquad n\to \infty
    \]
   and the lower endpoint admits the following explicit $\halflines$-dependent lower bounds: 
    \[
        \lambda_n^{\mathrm X,(0)}
        \ge
        \begin{cases}
            (2n-1)b_0\halflines^{-1},
            & \mathrm X=\Dirichlet,\\[1em]
            2(n-1)b_0 \halflines^{-1},
            & \mathrm X=\Neumann
        \end{cases}
    \]
    for all $n\in \bbN$.
\end{theorem}

\begin{proof}
    The chain of inequalities and the upper bound behaviour follow from Proposition~\ref{prop:spectrum-block-localization} and Theorem~\ref{thm:OU-half-line-asymptotics} respectively. On the other hand, as $b_0\le g_{\Graph}\le \halflines$, the lower bound can be deduced by Lemma~\ref{lem:half-line-domain-monotonicity} together with Remark~\ref{rem:rayleigh-comparison}.
\end{proof}

\appendix

\section{Ornstein--Uhlenbeck Operator on
\texorpdfstring{$[a,\infty)$}{[a,∞)}}
    \label{appendix:half-line}

In this appendix, for $a\ge 0$, we seek the eigenvalues of the Ornstein--Uhlenbeck operator defined by $\calL_a u = -\frac{1}{2}u'' + x u'$ on $L^2_\gamma([a,\infty))$, where
\[
    \gamma(\dxShort x)
    \coloneqq
    \varrho(x)\dx x
    =\frac{1}{\sqrt\pi}e^{-x^2}\dx x.
\]
subject to the Dirichlet and Neumann boundary conditions at $x=a$, i.e., with domains
\begin{align*}
    \domDir{\calL_a} & = \{ u\in H^2_{\gamma}([a, \infty)) : u(a)=0\}
    \quad\text{and}\\
    \domNeu{\calL_a} & = \{ u\in H^2_{\gamma}([a, \infty)) : u'(a)=0\},
\end{align*}
respectively. For convenience, we write $\calL_a^{\Dirichlet}\coloneqq (\calL_a,\domDir{\calL_a})$ and $\calL_a^{\Neumann}\coloneqq (\calL_a,\domNeu{\calL_a})$.
Because the operators are essentially a gauge-transformed quantum harmonic oscillator, their spectrum is discrete, bounded below, and diverges to infinity. 

\begin{proposition}
    \label{prop:half-line-characteristic-equation}
    Let $D_{\lambda}(\argument)$ denote the parabolic cylinder functions. Then
    \begin{align*}
        \spec(\calL_a^{\Dirichlet})=\left\{ \lambda>0: D_{\lambda}(\sqrt 2 a)=0\right\}
        \quad \text{and}\quad
        \spec(\calL_a^{\Neumann})=\{0\} \cup \left\{\lambda>0: D_{\lambda-1}(\sqrt 2 a)=0\right\}.
    \end{align*}
\end{proposition}

\begin{proof}
    Let $u\in H^2_{\gamma}((a, \infty))$ be an eigenfunction of 
    $\calL_a^{\mathrm X}$ 
    corresponding to an eigenvalue $\lambda$; where $\mathrm X \in \{\Dirichlet, \Neumann\}$.
    Since $-\calL_a^{\mathrm X}$ generates a contractive semigroup on $L^2_\gamma((a,\infty))$,
    we have $\lambda\ge0$.
    Introduce the transformation
    \[
        w(x)=\sqrt{\varrho(x)} u(x).
    \]
    A direct computation shows that $u$ solves $\calL_a u =\lambda u$ if and only if $w$ satisfies the Weber's equation
    \[
        w''(x) + (2\lambda+1 - x^2) w(x)=0.
    \] 
    Upon rescaling $z=\sqrt{2}x$, this becomes the standard parabolic cylinder equation
    \[
        w''(z) + \left(\lambda+\frac12 - \frac{z^2}{4}\right) w(z) = 0.
    \]
     To satisfy the growth condition at infinity, the unique linearly independent solutions (up to a constant) are:
     \[
        u_\lambda(x)= \varrho(x)^{-1/2} D_\lambda(\sqrt{2}x).
     \]
     In particular, $u_\lambda(a)=0 \Leftrightarrow D_\lambda(\sqrt{2}a)=0$, yielding the assertion in the Dirichlet case
     (note that $D_0(\sqrt 2 a) = e^{-a^2/2} \ne 0)$.
     On the other hand, for the Neumann boundary condition $u'_\lambda(a) = 0$, the chain rule along with the derivative property of $D_\lambda$ yields the characteristic equation
     \[
        \lambda D_{\lambda-1}(\sqrt{2}a) = 0,
     \]
     concluding the proof of the Neumann case.
\end{proof}

We now state the asymptotic behaviour of the eigenvalues.

\begin{theorem}
    \label{thm:OU-half-line-asymptotics}
    Denote the eigenvalues of $\calL_a^\Dirichlet$ and $\calL_a^\Neumann$ (ordered increasingly) by 
    \[
        \lambda_n^{\Dirichlet}(a)
        \quad\text{and}\quad 
        \lambda_n^{\Neumann}(a),
        \qquad 
        (n\in \bbN)
    \]
    respectively. Then, for fixed $a\ge 0$, we have
    \[
        \lambda_n^{\Dirichlet}(a)
        \sim 2n -1 + \frac{4a}{\pi}\sqrt{n} + \frac{4a^2}{\pi^2}
        \quad\text{and}\quad
        \lambda_n^{\Neumann}(a)
        \sim 2(n-1) + \frac{4a}{\pi}\sqrt{n} + \frac{4a^2}{\pi^2}
    \]
    as $n\to \infty$.
\end{theorem}

The leading terms agree with the harmonic-oscillator levels on the
half-line. The dependence on the endpoint $a$ enters through the
large-order asymptotics of the zeros of the parabolic cylinder
functions; see \cite{AliliPatiePedersen2005}.
Moreover, for $a=0$ we obtain the well-known spectrum of both realisations of the Ornstein-Uhlenbeck operator on $L^2_{\gamma}(\mathbb{R}_+)$; see~\cite[Lemma~4.7]{MugnoloRhandi2022}.

\begin{proof}[Proof of Theorem~\ref{thm:OU-half-line-asymptotics}]
    We provide the proof for the Neumann case. The computations for the Dirichlet case are analogous.

    By Proposition~\ref{prop:half-line-characteristic-equation}, the eigenvalues of $\calL_a^{\Neumann}$ are determined by the equation $\lambda D_{\lambda - 1}(\sqrt 2 a)=0$. So $\lambda_1^{\Neumann}=0$, and we now consider $n\ge 2$.
    Using the standard WKB (Wentzel--Kramers--Brillouin) approximation for a large order $\nu = \lambda - 1$ and a fixed argument $z$, cf. \cite[page~979]{AliliPatiePedersen2005}, the parabolic cylinder function is dominated by its oscillatory cosine phase:
    \[
        D_{\nu}(z) \sim C \cos\left(z\sqrt{\nu +\frac{1}{2}} - \frac{\pi\nu}{2}\right)
        ;
    \]
    where $C\equiv C(\nu)\neq 0$ is an amplitude factor.
    Hence the zeros of $D_\nu(z)$ are asymptotically determined
    for fixed $z$ and large $m$,
    by the phase condition
    \[
        z\sqrt{\nu+\tfrac12}-\frac{\pi\nu}{2}
        =
        -\left(m+\frac12\right)\pi
        .
    \]
    Substituting $z=\sqrt{2}a$ and simplifying gives
    \[
        \nu-\frac{2a}{\pi}\sqrt{2\nu+1}=2m+1.
    \]
    
    To solve for $\nu$, we substitute $y = \sqrt{2\nu + 1}$ into the previous equation that gives a standard quadratic with respect to $y$:
    \[
        y^2 - \frac{4a}{\pi}y - (4m + 3) = 0.
    \]
    Applying the quadratic formula and taking the positive branch yields
    \[
        y = \frac{2a}{\pi} + \sqrt{\frac{4a^2}{\pi^2} + 4m + 3}.
    \]
    Consequently,
    \[
        \nu = \frac{y^2 - 1}{2} = 
            2m+1+\frac{4a^2}{\pi^2} +\frac{2a}{\pi}\sqrt{4m+3+\frac{4a^2}{\pi^2}}.
    \]
    For large $m$, the Taylor approximation $\sqrt{4m + \ldots} \sim 2\sqrt{m}$ yields
    \[
        \nu \sim 2m+1+\frac{4a\sqrt{m}}{\pi}+\frac{4a^2}{\pi^2}.
    \]
    Recalling $\lambda=\nu+1$ and re-indexing by $m=n-2$, we arrive at
    \[
        \lambda \sim 2(n-1) + \frac{4a\sqrt{n-2}}{\pi}+\frac{4a^2}{\pi^2}.
    \]
    Finally, the estimate $\sqrt{n - 2} \sim \sqrt{n}$ for large $n$, yields
    the desired approximation of $\lambda_n^{\Neumann}$.
\end{proof}

\section*{Acknowledgements} 
Part of the work was done during very pleasant visits of the first author at the University of Salerno and University of Ljubljana.

 This article is based upon work
 from COST Actions 18232 MAT-DYN-NET and 24122 mSPACE, supported by COST (European Cooperation in Science and
 Technology), \url{www.cost.eu}.

\bibliographystyle{plainurl}
\bibliography{literature}

\end{document}